\documentclass[12pt,reqno,twoside]{amsart}
\usepackage[left=1 in,top=1 in,right=1 in,bottom=1 in]{geometry}
\usepackage{amssymb,amsfonts,amsthm,amsmath}
\usepackage{mathrsfs}
\usepackage{enumitem}
\usepackage{color}
\usepackage{hyperref}
\usepackage{cite}
\usepackage{enumitem}
\usepackage{mathtools}
\mathtoolsset{showonlyrefs}

\numberwithin{equation}{section}

\newcommand{\R}{\mathbb R}

\def\eqdef{\stackrel{\rm def}{=}}

\def\beq{\begin{equation}}
\def\eeq{\end{equation}}

\def\beqs{\begin{equation*}}
\def\eeqs{\end{equation*}}

\newtheorem{theorem}{Theorem}[section]
\newtheorem{lemma}[theorem]{Lemma}

\newtheorem{corollary}[theorem]{Corollary}
\newtheorem{assumption}[theorem]{Assumption}

\theoremstyle{definition}
\newtheorem{remark}[theorem]{Remark}

\newcommand{\tnum}{\rm(\roman*)}
\newcommand{\rnum}{\rm(\alph*)}

\definecolor{darkred}{rgb}{.70,.12,.20}

\definecolor{darkgreen}{rgb}{.20,.52,.14}

\newcommand{\esssup}{\mathop{\mathrm{ess\,sup}}}

\newcommand{\varep}{\varepsilon}

\newcommand{\eR}{\vec k}

\numberwithin{equation}{section}

\title{Studying a doubly nonlinear model of slightly compressible Forchheimer flows in rotating porous media}

\date{\today}

\makeatletter
\@namedef{subjclassname@2020}{%
  \textup{2020} Mathematics Subject Classification}
\makeatother


\begin{document}
\author{Emine Celik$^{1}$}
\address{$^{1}$Department of Mathematics, Sakarya University\\
54050, Sakarya, Turkey}
\email{eminecelik@sakarya.edu.tr}

\author{Luan Hoang$^{2}$}
\address{$^2$Department of Mathematics and Statistics,
Texas Tech University\\
1108 Memorial Circle, Lubbock, TX 79409--1042, U. S. A.}
\email{luan.hoang@ttu.edu}

\author{Thinh Kieu$^{3}$}
\address{$^{3}$Department of Mathematics, University of North Georgia, Gainesville Campus\\
3820 Mundy Mill Rd., Oakwood, GA 30566, U. S. A.}
\email{thinh.kieu@ung.edu}

\begin{abstract}
We study the generalized Forchheimer flows of slightly compressible fluids in rotating porous media. 
In the problem's model, the varying density in the Coriolis force is fully accounted for without any simplifications. 
It results in a doubly nonlinear parabolic equation for the density.
We derive \emph{a priori} estimates for the solutions in terms of the initial, boundary data and physical parameters, emphasizing on the case of unbounded data. 
Weighted Poincar\'e--Sobolev inequalities suitable to the equation's nonlinearity, adapted Moser's iteration and maximum principle are used and combined to obtain different types of estimates.
\end{abstract}

\maketitle
\tableofcontents

\pagestyle{myheadings}\markboth{E. Celik, L. Hoang,  and T. Kieu}
{Doubly Nonlinear Model of Forchheimer Flows in Rotating Porous Media}

\section{Introduction}\label{intro}
We continue the investigation of the Forchheimer flows of slightly compressible fluids in rotating porous media, which was initiated in our previous work \cite{CHK3}. In paper \cite{CHK3}, we simplified the Coriolis force's dependence on the density in the model  in order to reduce the complexity of the problem. The resulting partial differential equation (PDE) was of degenerate parabolic type and we were able to understand its key nonlinear structure, and derived various estimates for its solutions.
In this paper, we study the full model without any simplifications. As we will see, the PDE becomes a doubly nonlinear parabolic equation. We will analyze this more complicated  equation in more general context by realizing its new structure and utilizing other techniques with appropriate adaptations and improvements.

We consider a porous medium, with constant porosity $\tilde\phi\in(0,1)$ and constant permeability $k>0$, rotated with a constant angular velocity $\tilde\Omega \eR$,  where $\tilde\Omega\ge 0$ is the constant angular speed, and $\eR$ is a constant unit vector.
We study the dynamics of fluid flows in this porous medium.

The equation for the Darcy flows in rotating porous media written in a rotating frame is, see Vadasz \cite{VadaszBook},
\beq\label{DarcyRot}
\frac{\mu}{k}v+ \frac{2\rho \tilde\Omega}{\tilde\phi} \eR \times v + \rho \tilde\Omega^2 \eR \times (\eR \times x)=-\nabla p+\rho \vec g,
\eeq
where 
$\mu$ is the dynamic viscosity,
$v$ is the velocity, 
$\rho$ is the fluid density, 
$p$ is the pressure, 
$x$ is the position in the rotating frame,
$\vec g$ is the gravitational acceleration,
$\tilde\Omega^2 \eR \times (\eR \times x)$ is centripetal acceleration,
and $(2\rho \tilde\Omega/\tilde\phi) \eR \times v$ represents the Coriolis effects in the rotating porous medium.

For fluid flows that obey Forchheimer's two-term law, we have 
 \beq\label{Ftwo}
\frac{\mu}{k} v+ \frac{c_F \rho}{\sqrt{k}}|v|v + \frac{2\rho \tilde\Omega}{\tilde\phi} \eR \times v+\rho \tilde\Omega^2 \eR \times (\eR \times x)=-\nabla p +\rho \vec g,
 \eeq
where  $c_F>0$ is the Forchheimer constant \cite{Ward64}. Other equations for Forchheimer's three-term and power laws can be obtained similarly. 

Equations \eqref{DarcyRot} and \eqref{Ftwo} can be written in one general form, namely, the generalized Forchheimer  equation in rotating porous media
 \beq\label{FM}
 \sum_{i=0}^N a_i \rho^{\bar\alpha_i} |v|^{\bar\alpha_i} v+ \frac{2\rho \tilde\Omega}{\tilde\phi} \eR \times v+\rho \tilde\Omega^2 \eR \times (\eR \times x)=-\nabla p +\rho \vec g.
 \eeq
 
Regarding the first sum in equation \eqref{FM}, the dependence on the density is expressed by the term $\rho^{\bar\alpha_i}$ which is obtained by using Muskat's dimension analysis \cite{MuskatBook}.

For the Forchheimer equations and other related models of fluid flows in porous media that differ from the ubiquitous Darcy's law, 
the interested reader is referred to the books \cite{BearBook,NieldBook,StraughanBook}.
Regarding their mathematical analysis in the case without rotation, see \cite{StraughanBook,Zabensky2015a,ChadamQin,Payne1999a,Payne1999b,MTT2016,CKU2006,KR2017} for incompressible fluids,  see \cite{ABHI1,HI2,HIKS1,HKP1,HK2,CHK1,CHK2,CH1,CH2} for compressible fluids, and references therein. 
For more information about fluid flows in rotating porous media, see \cite{VadaszBook} and, also, our previous mathematical study \cite{CHK3}.

Hereafter, we fix the integer $N\ge 1$, the powers $\bar\alpha_0=0<\bar\alpha_1<\bar\alpha_2<\ldots<\bar\alpha_N$, and positive constant  coefficients $a_0,a_1,\ldots,a_N$.

Define a function $g:\mathbb{R}^+\rightarrow\mathbb{R}^+$  by
\beq\label{eq2}
g(s)=a_0 + a_1s^{\bar \alpha_1}+\cdots +a_Ns^{\bar \alpha_N}=\sum_{i=0}^N a_i s^{\bar \alpha_i}\quad\text{for } s\ge 0.
\eeq 
In \eqref{eq2} and throughout the paper, we conveniently use $0^0=1$.

Set $\mathcal R(\rho)=2\rho \tilde\Omega/\tilde\phi$.
Multiplying both sides of  \eqref{FM}  by $\rho$ gives
 \beq\label{neweq1}
 g(|\rho v|) \rho v + \mathcal R(\rho) \eR \times (\rho v)  =-\rho\nabla p+ \rho^2 \vec g- \rho^2 \tilde\Omega^2 \eR \times (\eR \times x).
 \eeq 

We solve for $\rho v$ from \eqref{neweq1} in terms of the vector on its right-hand side and the $\mathcal R(\rho)$. 
To do that, we  define the function $F_z:\R^3\to\R^3$, for any $z\in\R$, by
\beq\label{Fdef}
F_z(v)= g(|v|)v + z  \mathbf J v \quad \text{ for }v\in\R^3,
\eeq
where $\mathbf J$ is the $3\times 3$ matrix for which $\mathbf J x=\eR \times x$ for all $x\in\R^3$.

Equation \eqref{neweq1} is rewritten as
\beq\label{FRv}
F_{\mathcal R(\rho)}(\rho v)=-(\rho\nabla p - \rho^2 \vec g + \rho^2 \tilde\Omega^2 \mathbf J^2 x).
\eeq

Thanks to \cite[Lemma 1.1]{CHK3}, the function $F_z$ is odd and bijective for each $z\in\R$.
Then we can invert \eqref{FRv} to have
 \beq\label{new2}
\rho v= -F_{\mathcal R(\rho)}^{-1}(\rho\nabla p - \rho^2 \vec g + \rho^2 \tilde\Omega^2 \mathbf J^2 x).
\eeq 

In article \cite{CHK3}, $\mathcal R(\rho)$ was approximated by a constant $\mathcal R=2\rho_* \tilde\Omega/\tilde\phi$, for some constant density  $\rho_*$. This resulted in a simpler equation than \eqref{new2}. On contrary, we will keep the dependence of $\mathcal R(\rho)$ on $\rho$ in the current paper, and treat equation \eqref{new2} in that original form.

We recall that the fluid's compressibility for isothermal conditions is
\beqs
\varpi=-\frac{1}{V}\frac{dV}{dp}=\frac{1}{\rho}\frac{d\rho}{dp},
\eeqs
where $V$, here, denotes the fluid's volume. In many cases such as (isothermal) compressible liquids, $\varpi$ is assumed to be a constant \cite{MuskatBook,BearBook}. In particular, it is a small positive constant for (isothermal) slightly compressible fluids such as crude oil and water. This condition is commonly used in petroleum and reservoir engineering \cite{AhmedHandbook2nd,Dakebook}, where the fluid dynamics in porous media have important applications. The current paper is focused on (isothermal)  slightly compressible fluids, hence, we study the following equation of state 
   \beq\label{slight}
\frac{1}{\rho}   \frac{d\rho}{dp}=\varpi, 
\quad \text{where the constant compressibility $\varpi>0$ is small}.
   \eeq

The equation of continuity is
\beq\label{eq5}
\tilde \phi\frac{\partial \rho}{\partial t} +\nabla\cdot(\rho v)=0.
\eeq

Note, by \eqref{slight}, that $\rho\nabla p=\varpi^{-1}\nabla \rho$.
Then combining \eqref{eq5} with \eqref{new2},  we obtain 
\beq\label{eq0}
\tilde \phi\frac{\partial \rho}{\partial t}=\nabla\cdot(F_{\mathcal R(\rho)}^{-1}(\varpi^{-1} \nabla \rho - \rho^2 \vec g+ \rho^2 \tilde \Omega^2 \mathbf J^2 x)).
\eeq

The gravitational field in the rotating frame is 
$\vec g(t)=-\tilde{\mathcal G} \tilde e_0(t)$, where $\tilde{\mathcal G}>0$ is the gravitational constant, and $\tilde e_0\in C^\infty(\R,\R^3)$ with  $|\tilde e_0(t)|=1$ for all $t\in\R$.

We make a simple change of variable  $u=\rho/\varpi$, and corresponding scaling of parameters
\beqs 
\phi =\varpi\tilde \phi>0,\quad
\mathcal G=\varpi^2\tilde{\mathcal G},\quad \Omega=\varpi \tilde \Omega.
\eeqs

Note that
\beq\label{Rstar}
\mathcal R(\rho)=R_* u, \text{ where }R_*=2\varpi\tilde\Omega/\tilde\phi=2\varpi\Omega/\phi.
\eeq

Then we obtain from \eqref{eq0} that
\beq\label{ueq}
\phi \frac{\partial u}{\partial t}=  \nabla \cdot\left (X\left(u,\nabla u+u^2[-\mathcal G \tilde e_0(t)+ \Omega^2  \mathbf J^2 x]\right)\right),
\eeq
where
\beq\label{Xdef}
X(z,y)=F_{R_* z}^{-1}(y) \text{ for }z\in\R,\ y\in\R^3.
\eeq

By making another transformation $\tilde u(x,t)=u(x,\phi t)$ and rewriting equation \eqref{ueq} for $\tilde u(x,t)$ and then removing the tilde notation, we obtain
\beq\label{mainuX}
\frac{\partial u}{\partial t}=  \nabla \cdot\left (X\left(u,\nabla u+u^2 \mathcal Z(x,t)\right)\right),  
\eeq
where
\beq\label{Zx1}
\mathcal Z(x,t)=-\mathcal G e_0(t)+ \Omega^2  \mathbf J^2 x\text{ with } e_0(t)=\tilde e_0(\phi t).
\eeq

We will focus on the Dirichlet boundary condition for $u(x,t)$.
Let $U$ be an open, bounded set in $\R^3$ with $C^1$ boundary $\Gamma=\partial U$.
We study the initial boundary value problem (IBVP)
\beq\label{ibvpg}
\begin{aligned}
\begin{cases}
 \begin{displaystyle}\frac{\partial u}{\partial t}\end{displaystyle}=\nabla\cdot \left(X\left(u,\nabla u+u^2 \mathcal Z(x,t)\right)\right)\quad &\text{in}\quad U\times (0,\infty)\\
u(x,0)=u_0(x) \quad &\text{in}\quad U\\
u(x,t)= \psi(x,t)\quad &\text{in}\quad \Gamma\times (0,\infty),
\end{cases}
\end{aligned}
\eeq
where $u_0(x)$ and $\psi(x,t)$ are given. 

In previous article \cite{CHK3}, the maximum estimates for the solutions are achieved by the use of the maximum principle. This method requires the initial data to be bounded. In this paper, we aim at treating also unbounded initial data. For that, we will use the Moser iteration. 
Regarding the newly obtained PDE \eqref{mainuX}, it  has extra dependence on $u$, in addition to $\nabla u+u^2 \mathcal Z(x,t)$.
This dependence turns out to yield new weights, which depend on the solution $u$ itself, in the energy estimates. Therefore, more technical treatments are required.
Indeed, we establish suitable weighted Poincar\'e--Sobolev inequalities  to deal with these weights.
We are then able to estimate the Lebesgue norms of the solutions, and, by the Moser iteration, their essential supremum. 
These short-time estimates are combined with the maximum principle to give all time estimates.
Moreover, we highlight that our estimates are derived by appropriately handled techniques to provide explicit dependence on physical parameters including the angular speed of the rotation.

The paper is organized as follows. 
 In section~\ref{Prem},  we present crucial properties  of the function  $X(z,y)$ by recasting the corresponding results in \cite{CHK3} but with explicit dependence on $z$, see Lemmas \ref{lem21} and \ref{Xder}. We also establish some elliptic and parabolic Poincar\'e-Sobolev inequalities with certain weights. These particular inequalities are then formulated in suitable forms for our treatment of the double nonlinearity in \eqref{mainuX},
 see Lemma \ref{PS1}, Corollary \ref{PS2} and Lemma \ref{PS3}.
In section~\ref{L-p-est}, we study the IBVP \eqref{ubar0} for $\bar{u}(x,t)$, which, briefly speaking, is a nonnegative solution $u(x,t)$ of \eqref{ibvpg} shifted by the boundary data. We obtain the $L^\alpha$-estimates, for sufficient large $\alpha\in(0,\infty)$, for $\bar u$ in terms of the initial and boundary data, see Theorem~\ref{aprio1}. We also establish  in  Theorem~\ref{aprio1} a weighted  $L_{x,t}^{2-a}$-estimate for the gradient of $\bar{u}$, with the number $a\in(0,1)$ defined in \eqref{axi} and the weight function depending on the solution $u$. 
Section~\ref{maxsec} is focused on the $L^\infty$-estimates for $\bar{u}$. 
By adapting Moser's iteration, we derive,  in Theorem~\ref{maxest}, an upper bound for $\bar u$'s $L^\infty$-norm expressed in terms of its $L^\alpha$-norm for some finite number $\alpha>0$. 
The main estimate, for small time $t>0$, is then obtained in Theorem~\ref{thm45}  in terms of certain $L^\alpha$-norms of the initial and boundary data. 
All estimates' dependence on the physical parameters is expressed via the number $\chi_*$, see \eqref{chidef}. It is meticulously tracked in each step of the complicated iteration.
In section~\ref{maxprin}, we establish the maximum principle for classical solutions of \eqref{mainuX} in Theorem~\ref{maxpr}. 
Combining this maximum principle with the short-time estimates in section \ref{maxsec}, we obtain the maximum estimates in Theorem~\ref{maxestsol} for nonnegative solutions of the IBVP \eqref{ibvpg} for all time $t>0$ even when the initial data is unbounded. 

\section{Preliminaries}\label{Prem}

\subsection{Notation} 
A vector $x\in\R^n$ is denoted by a $n$-tuple $(x_1,x_2,\ldots,x_n)$ and considered as a column vector, i.e., a $n\times 1$ matrix. Hence $x^{\rm T}$ is the $1\times n$ matrix $(x_1\ x_2\ldots x_n)$.

For two vectors $x,y\in \R^n$, their dot product is $x\cdot y=x^{\rm T}y=y^{\rm T}x$, while $xy^{\rm T}$ is the $n\times n$ matrix $(x_iy_j)_{i,j=1,2,\ldots,n}$.

Let $\mathbf A=(a_{ij})$ and $\mathbf B=(b_{ij})$ be any $n\times n$ matrices of real numbers. Their inner product is
\beqs 
\mathbf A:\mathbf B\eqdef {\rm trace}\left(\mathbf A\mathbf B^{\rm T}\right)=\sum_{i,j=1}^n a_{ij}b_{ij}.
\eeqs

The Euclidean norm of the matrix $\mathbf A$ is $$|\mathbf A|=(\mathbf A:\mathbf A)^{1/2}=\left(\sum_{i,j=1}^n a_{ij}^2\right)^{1/2}.$$ 
(Note that we do not use $|\mathbf A|$ to denote the determinant in this paper.)

When $\mathbf A$ is considered as a linear operator, another norm is defined by
\beqs
\|\mathbf A\|_{\rm op}=\max\left\{ \frac{|\mathbf Ax|}{|x|}:x\in\R^n,x\ne 0\right\} = \max\{ |\mathbf Ax|:x\in\R^n,|x|=1\}.
\eeqs

It is well-known that
\beq\label{nnorms}
\|\mathbf A\|_{\rm op}\le |\mathbf A|\le c_*\|\mathbf A\|_{\rm op},
\eeq
where $c_*=c_*(n)$ is a positive constant independent of $\mathbf A$.

Clearly, the matrix $\mathbf J$ in \eqref{Fdef} satisfies 
\beq\label{Jineq}
|\mathbf Jx|\le|\vec{k}||x|=|x|  \text{ and } |\mathbf J^2x|\le|\mathbf Jx|\le |x| \text{ for all }x\in\mathbb{R}^3.
\eeq

For a function $f=(f_1,f_2,\ldots,f_m):\R^n\to\R^m$, its derivative is the $m\times n$ matrix
\beq\label{deriv}
Df=\Big(\frac{\partial f_i}{\partial x_j}\Big)_{1\le i\le m,1\le j\le n}.
\eeq

In particular, when  $m=1$, i.e.,  $f:\R^n\to \R$, the derivative is 
$$Df=\left (\frac{\partial f}{\partial x_1}\quad 
\frac{\partial f}{\partial x_2}\quad \ldots \quad 
\frac{\partial f}{\partial x_n} \right),$$ 
while its gradient vector is  
$\nabla f=(\frac{\partial f}{\partial x_1},\frac{\partial f}{\partial x_2},\ldots,\frac{\partial f}{\partial x_n} )=(Df)^{\rm T}.$

The Hessian matrix is $$D^2 f=D(\nabla f)=  \Big(\frac{\partial^2 f}{\partial x_j \partial x_i} \Big)_{i,j=1,2,\ldots,n}.$$

We also write $D_xf$ for $Df$ in \eqref{deriv} in case the variables need to be indicated explicitly.

\subsection{Auxiliary inequalities}
The following is a convenient consequence of Young's inequality. If $x_i\ge 0$ and $z_i>1$ for $i=1,2,\ldots,k$ with $k\ge 2$ such that $\sum_{i=1}^k 1/z_i=1$, then
\beq\label{eeY}
\prod_{i=1}^k x_i\le \sum_{i=1}^k x_i^{z_i}.
\eeq
For the sake of brevity, we call \eqref{eeY} Young's inequality in this paper.

For $x,y\ge 0$, one has 
\beq\label{ee4}
x^\beta \le x^\alpha+x^\gamma \text{ for } 0\le \alpha\le \beta\le\gamma,
\eeq
\beq\label{ee3}
 (x+y)^p\le 2^{(p-1)^+}(x^p+y^p)  \text{ for }  p>0,
\eeq
where $z^+=\max\{z,0\}$ for any $z\in \R$.
We also frequently use the following  alternative form  of \eqref{ee3}
\beq\label{ee2}
(x+y)^p\le 2^p(x^p+y^p)  \text{ for all } x,y\ge 0,\ p>0.
\eeq

By the triangle inequality and  inequality \eqref{ee3}, we have
\beq\label{ee6}
|x\pm y|^p\ge 2^{-(p-1)^+}|x|^p-|y|^p \quad \text{for all } x,y\in\R^n,\quad p>0.
\eeq

The interpolation inequality for the Lebesgue integrals: if  $0<p<s<q$ and $1/s=\theta/p+(1-\theta)/q$ for $\theta\in(0,1)$, then  
\beqs
\left(\int |f|^sd\mu\right)^\frac{1}{s}
\le \left(\int |f|^p d\mu\right)^\frac{\theta}{p}
\left(\int |f|^q d\mu\right)^\frac{1-\theta}{q}.
\eeqs

\subsection{Characteristics of the function $X(z,y)$}

Note that $v=\widetilde X(z,y)\eqdef F_z^{-1}(y)$ is the unique solution of the equation 
 $$G(z,y,v)\eqdef F_z(v)-y=g(|v|)v + z  \mathbf J v -y=0 \text{ for $z\in\R,y,v\in\R^3$.}$$

The partial derivatives of $G$ are  
\begin{align*}
D_vG(z,y,v)&=DF_z(v)=g'(|v|)\frac{v v^{\rm T}}{|v|}+g(|v|)\mathbf I_3+ z\mathbf J\text{ for }v\ne 0,\\
D_vG(z,y,0)&=DF_z(0)=g(0)\mathbf I_3+z \mathbf J,\\
D_zG(z,y,v)&=\mathbf Jv, \quad D_yG(z,y,v)=-\mathbf I_3.
\end{align*}

One can verify that $G\in C^1(\R^7)$. Same as in \cite[Lemma 2.3]{CHK3}, $D_vG$ is invertible on $\R^7$. By the Implicit Function Theorem, the solution $v=\widetilde X(z,y)$ belongs to $C^1(\R^4)$.
Consequently, the function $X(z,y)=\widetilde X(R_*z,y)$ belongs to $C^1(\R^4)$.

Throughout the paper, we denote 
\beq\label{axi}
a=\frac{\bar\alpha_N}{1+\bar\alpha_N}\in(0,1),\quad 
\chi_0=g(1)=\sum_{i=0}^N a_i.
\eeq 

The properties of the function $X(z,y)$, which is defined by \eqref{Xdef}, are similar to those established in \cite[Lemmas 2.1 and 2.4]{CHK3}.
Now that $X$ depends on $z$, we need some explicit dependence on $z$ for the inequalities there.
In fact, thanks to \eqref{Rstar}, we can replace $\chi_1=\chi_0+\mathcal R$ in \cite[Lemmas 2.1 and 2.4]{CHK3} with $\chi_0+R_*z$, and, hence, rewrite those two lemmas as Lemmas \ref{lem21} and \ref{Xder} below. Denote $\R_+=[0,\infty)$.

\begin{lemma}\label{lem21}

{\rm (i)} One has 
\beq\label{X0}
\frac{c_1(\chi_0+R_*z)^{-1}|y|}{(1+|y|)^a}\le |X(z,y)|\le \frac{c_2(\chi_0+R_*z)^a|y|}{(1+|y|)^a} \text{ for all }z\in\R_+,y\in \mathbb R^3,
\eeq
where 
$c_1=\min\{1,\chi_0\}^a $ and $c_2=2^a c_1^{-1}\min\{a_0,a_N\}^{-1}$.
Alternatively,
\beq\label{X1}
 (\chi_0+R_*z)^{-(1-a)}|y|^{1-a}-1\le |X(z,y)|\le c_3|y|^{1-a} \text{ for all }z\in\R_+,y\in \mathbb R^3,
 \eeq
where $c_3=(a_N)^{a-1}$.

{\rm (ii)} One has
\beq\label{X2}
\frac{c_4 (\chi_0+R_*z)^{-2} |y|^2}{(1+|y|)^a}  \le X(z,y)\cdot y\le \frac{c_2 (\chi_0+R_*z)^a |y|^2}{(1+|y|)^a} \text{ for all }z\in\R_+,y\in \mathbb R^3,
\eeq
where $c_4=(\min\{1,a_0,a_N\}/2^{\alpha_N})^{1+a}$.
Alternatively,
\beq\label{X3}
c_5(\chi_0+R_*z)^{-2} (|y|^{2-a}-1)  \le X(z,y)\cdot y\le c_3|y|^{2-a}\text{ for all }z\in\R_+,y\in \mathbb R^3,
\eeq
where $c_5=2^{-a}c_4$.
\end{lemma}

Although inequalities \eqref{X0} and \eqref{X2} provide more precise dependence on $|y|$ than \eqref{X1} and \eqref{X3}, the latter two are sufficient and more convenient in this paper.

\begin{lemma}\label{Xder}
For all $z\in\R_+$ and $y\in \mathbb R^3$, 
the matrix $D_yX$ of partial derivatives in the variable $y$ satisfies
\beq\label{Xprime}
c_6(\chi_0+R_*z)^{-1}(1+|y|)^{-a}\le |D_y X(z,y)|\le  c_7(1+\chi_0+R_*z)^a (1+|y|)^{-a} ,
\eeq
\beq\label{hXh} 
\xi^{\rm T} D_y X(z,y)\xi \ge c_8 (\chi_0+R_*z)^{-2}(1+|y|)^{-a} |\xi|^2 \text{ for all }\xi\in\R^3,
\eeq
where 
\beqs
c_6= \sqrt 3(2^{-\alpha_N}\min\{1,a_N\})^a/(\alpha_N+2),\
c_7=c_* 2^{\alpha_N}/\min\{a_0,a_N\},\
c_8=c_4/(\alpha_N+2)^2
\eeqs
with $c_*=c_*(3)$ given in \eqref{nnorms}.
\end{lemma}

To complement the estimate of $D_yX$ in \eqref{Xprime}, we derive in \eqref{DX2} and \eqref{DX3} below some estimates for $D_zX$.
Taking the partial derivative in $z$ of the equation $G(R_*z,y,X(z,y))=0$ we have
\beqs
0=R_* \mathbf JX(z,y)+DF_{R_*z}(X(z,y))D_zX(z,y)=R_* \mathbf JX(z,y)+ (D_y X(z,y))^{-1}D_zX(z,y),
\eeqs
which implies
\beq\label{DzX}
D_zX(z,y)=- R_* D_yX(z,y) \mathbf JX(z,y).
\eeq

Combining formula \eqref{DzX} with estimates \eqref{Xprime} and \eqref{X0}, respectively, \eqref{X1}, yields
\beq\label{DX2}
|D_z X(z,y)|\le  c_2c_7R_*(1+\chi_0+R_*z)^{2a} |y|(1+|y|)^{-2a},
\eeq
respectively,
\beq\label{DX3}
|D_z X(z,y)|\le  c_3c_7R_*(1+\chi_0+R_*z)^{a} |y|^{1-a}(1+|y|)^{-a}.
\eeq

\subsection{Weighted Poincar\'e--Sobolev inequalities}

In this subsection, we consider the space $\R^n$, with $n\ge 2$, and an open, bounded set $U\subset \R^n$.
For a number $p\in[1,n)$, its Sobolev conjugate is $p^*=np/(n-p)$.
We establish some specific inequalities of Poincar\'e--Sobolev type with weight functions.

\begin{lemma}[Elliptic version]\label{PS1}
 Suppose $p$ and $r_*$ are  positive numbers that satisfy
 \beq\label{powcond} 
\frac{n}{n+p}<r_*<1\text{ and } 
\frac 1 p\le r_*<\frac n p.
 \eeq 

Let $r$, $s$ and  $\alpha$ be numbers such that
 \beq\label{alz} 
r>0,\ 
\alpha>0,\ 
\alpha\ge s\ge 0,
 \text{ and } \alpha> \frac{nr_*(r-p+s)}{r_*(n+p)-n}.
 \eeq 

Denote
\beq\label{mdef}
m=\frac{\alpha-s+p}{p}.
\eeq 

Let $u(x)$ be a function that vanishes on  $\partial U$ with $|u|^m\in W^{1,{r_*p}}(U)$, and $W(x)$ be a positive function on $U$.
Then one has, for any $\varep>0$, that
 \begin{equation}\label{ay2}
\begin{aligned}
 \int_U |u|^{\alpha+r} dx
  &\le \varep \int_U |u|^{\alpha-s}|\nabla u|^p W dx +\varep^{-\frac\theta{1-\theta}}(\bar c m)^\frac{\theta p}{1-\theta} \|u\|_{L^\alpha}^{\alpha+\mu} \|W^{-1}\|_{L^\frac{r_*}{1-r_*}}^\frac{\theta}{1-\theta},
\end{aligned}
 \end{equation}
where 
\beq \label{mt}
\theta=\frac{rn r_*}{nr_*(p-s)+\alpha(r_*(n+p)-n)}\in(0,1), \quad 
\mu=\frac{r+\theta(s-p)}{1-\theta}>-\alpha,
\eeq 
and positive constant $\bar c$, which appears in \eqref{Sobi} below, depends on $U$ and $r_*p$, but not on $u,W,r, \alpha, s$. 
\end{lemma}
\begin{proof}
We can use the calculations in the proof of Lemma 2.1(ii)  up to inequality (2.17) in \cite{CHK1}  applied to
 \beqs 
 \bar{p}:=r_* p, \quad \bar{\alpha}:=r_*\alpha,\text{ and } \bar{s}:=r_* s.
 \eeqs
 Then other numbers $m$ in \cite[(2.8)]{CHK1} and $q$ in \cite[(2.16)]{CHK1} become
 \beq\label{mqb}
 \bar{m}=\frac{\bar\alpha-\bar{s}+\bar{p}}{\bar{p}}=m\text{ and } 
 \bar{q}=\bar p^* \bar m=(r_* p)^* m=\frac{nr_*(\alpha-s+p)}{n-r_*p}.
 \eeq

Thanks to the last condition in \eqref{powcond} and the first condition in \eqref{alz}, one has 
$1\le \bar p<n$, $\bar\alpha\ge\bar s$ and $\bar m=m\ge 1$.
Because $\bar{m}\ge 1$ and $u$ vanishes on $\partial U$, we have the following Poincar\'e--Sobolev inequality for $|u|^{\bar{m}}$, which corresponds to inequality  \cite[(2.14)]{CHK1},
\beq\label{Sobi}
  \||u|^{\bar{m}}\|_{L^{\bar{p}^*}}\le \bar c \| \nabla(|u|^{\bar{m}})\|_{L^{\bar{p}}},
\eeq
where $\bar c>0$ depends on $U$ and $\bar p$.
Elementary calculations, see inequality \cite[(2.17)]{CHK1}, yield from \eqref{Sobi} that
 \begin{align*}
  \|u\|_{L^{\bar{q}}}&\le (\bar c\bar{m})^{1/\bar{m}} \left(\int_U |u|^{\bar{\alpha}-\bar{s}}|\nabla u|^{\bar{p}}dx\right)^{1/(\bar{\alpha}-\bar{s}+\bar{p})}\\
  &=(\bar c m)^\frac{1}{m} \left(\int_U \Big[|u|^{\alpha-s}|\nabla u|^p W(x)\Big]^{r_*} \cdot W(x)^{-r_*} dx\right)^\frac{1}{r_*(\alpha-s+p)}.
 \end{align*}

Denote $I=\int_U |u|^{\alpha-s}|\nabla u|^p W dx$ and note that $\alpha-s+p=mp$.
Applying H\"older's inequality with powers $1/r_*$ and $1/(1-r_*)$ to the last integral gives 
 \begin{equation}\label{229a}
  \|u\|_{L^{\bar{q}}}\le (\bar c m)^\frac{1}{m} I^\frac{1}{mp}\|W^{-1}\|_{L^\frac{r_*}{1-r_*}}^\frac1{mp}.
 \end{equation}

Thanks to the fact $r_*>n/(n+p)$, we have $r_*(n+p)-n>0$ and $nr_*>n-r_*p$, which, together with the last assumption in \eqref{alz},  yield
$$\alpha>\frac{nr_*r+nr_*(s-p)}{r_*(n+p)-n} > \frac{(n-r_*p)r+nr_*(s-p)}{r_*(n+p)-n}.$$
This implies $\alpha+r<\bar q$.

Because $\alpha<\alpha+r<\bar{q}$, we can find a number $\theta_0\in(0,1)$ such that
$$\frac 1{\alpha+r}=\frac{\theta_0}{\bar q}+\frac{1-\theta_0}{\alpha}.$$

In fact, $\theta_0$ is explicitly given by
\beq\label{deft0}
 \theta_0=\frac{r\bar{q}}{(\alpha+r)(\bar{q}-\alpha)}.
\eeq

Applying interpolation inequality  and combining it with \eqref{229a}, we have
\beq\label{u1}
 \int_U |u|^{\alpha+r} dx
 \le \left(\| u\|_{L^{\bar{q}}}^{\theta_0}\|u\|_{L^{\alpha}}^{1-\theta_0}\right)^{\alpha+r}
 \le (\bar c m)^\frac{\theta_0 (\alpha+r)}{m} I^\frac{\theta_0 (\alpha+r)}{mp} \|W^{-1}\|_{L^\frac{r_*}{1-r_*}}^\frac{\theta_0 (\alpha+r)}{mp} \|u\|_{L^{\alpha}}^{(1-\theta_0)(\alpha+r)}.
\eeq 

Denoting 
\beq \label{traw}
\theta=\frac{\theta_0 (\alpha+r)}{mp},
\eeq 
we rewrite \eqref{u1} as
\beq\label{ur}
 \int_U |u|^{\alpha+r} dx
\le \left(\varep^\theta I^\theta\right) \cdot \left(\varep^{-\theta} (\bar c m)^{\theta p} \|W^{-1}\|_{L^\frac{r_*}{1-r_*}}^\theta \|u\|_{L^{\alpha}}^{(1-\theta_0)(\alpha+r)}\right). 
\eeq 

Using the values of $m$, $\bar q$ and $\theta_0$ in \eqref{mdef}, \eqref{mqb} and \eqref{deft0}, respectively, we calculate the number $\theta$ in \eqref{traw} and find that it is the same as in \eqref{mt}.

By, again, the last assumption in \eqref{alz}, we have $\theta\in(0,1)$.
 Then applying  Young's inequality \eqref{eeY} with powers $1/\theta$ and $1/(1-\theta)$ to the product on the right-hand side of inequality \eqref{ur} gives
\beq\label{ay1}
 \int_U |u|^{\alpha+r} dx
\le \varep I  +\varep^{-\frac\theta{1-\theta}}(\bar c m)^\frac{\theta p}{1-\theta} \|W^{-1}\|_{L^\frac{r_*}{1-r_*}}^\frac{\theta}{1-\theta} \|u\|_{L^{\alpha}}^\frac{(1-\theta_0)(\alpha+r)}{1-\theta}.
\eeq

Recalculating the last power, with the use of identity in \eqref{traw}, we have
\beq\label{am} 
\frac{(1-\theta_0)(\alpha+r)}{1-\theta}
=\frac{(\alpha+r)-\theta mp}{1-\theta}
=\frac{(\alpha+r)-\theta(\alpha-s+p)}{1-\theta}=\alpha+\mu.
\eeq 
Then we obtain \eqref{ay2} from \eqref{ay1}. Since $\theta_0,\theta\in(0,1)$ in \eqref{am}, we have $\alpha+\mu>0$, which gives $\mu>-\alpha$ in \eqref{mt}.
 \end{proof}

Note from \eqref{powcond} that $p\ge 1/r_*>1$. Conversely, if $p>1$, then the set of $r_*$ that satisfies \eqref{rcond} is not empty.
 
\begin{corollary}\label{PS2} 
 Assume \eqref{powcond} and \eqref{alz}, and let $\bar c,m, \theta, \mu$ be defined as in Lemma \ref{PS1}. 
 
 Let $u(x)$ be as in Lemma \ref{PS1}, and $\varphi(x)$ be a function on $U$,  
and define $v=u+\varphi$. Let $\beta>0$ and $\varep>0$.
\begin{enumerate}[label=\tnum]
\item\label{i24}  Then one has
  \begin{equation}\label{ay0}
\begin{aligned}
 &\int_U |u|^{\alpha+r} dx
  \le \varep \int_U |u|^{\alpha-s}|\nabla u|^p (1+|v|)^{-\beta}dx\\
 &\quad +2^{\frac{\theta(1+(\beta-1)r_*)}{(1-\theta)r_*}}\varep^{-\frac\theta{1-\theta}}(\bar c m)^\frac{\theta p}{1-\theta} 
  \|u\|_{L^\alpha}^{\alpha+\mu}  \Big(\|u\|_{L^\frac{\beta r_*}{1-r_*}}^{\frac{\beta\theta}{1-\theta}}
  +\|1+|\varphi|\|_{L^\frac{\beta r_*}{1-r_*}}^\frac{\beta \theta}{1-\theta}\Big).
\end{aligned}
 \end{equation}
 
\item\label{ii24}  If, in addition,  
\beq\label{alph2} \alpha\ge \beta r_*/(1-r_*),
\eeq 
then one has
 \begin{equation}\label{ay}
\begin{aligned}
 &\int_U |u|^{\alpha+r} dx
  \le \varep \int_U |u|^{\alpha-s}|\nabla u|^p (1+|v|)^{-\beta}dx\\
 &+2^{\frac{\theta(1+(\beta-1)r_*)}{(1-\theta)r_*}}\varep^{-\frac\theta{1-\theta}}(\bar c m)^\frac{\theta p}{1-\theta} 
  \Big(|U|^{\frac{\theta(\alpha(1-r_*)-\beta r_*)}{\alpha r_*(1-\theta)}}\|u\|_{L^\alpha}^{\alpha+\mu+\frac{\beta\theta}{1-\theta}}+
  \|u\|_{L^\alpha}^{\alpha+\mu} \|1+|\varphi|\|_{L^\frac{\beta r_*}{1-r_*}}^\frac{\beta \theta}{1-\theta}\Big).
\end{aligned}
 \end{equation}
 \end{enumerate}
\end{corollary}
\begin{proof}
(i) Applying inequality \eqref{ay2} to $W(x)=(1+|v|)^{-\beta}$, we have 
 \begin{equation}\label{ay22}
\begin{aligned}
 \int_U |u|^{\alpha+r} dx
  &\le \varep \int_U |u|^{\alpha-s}|\nabla u|^p (1+|v|)^{-\beta}dx\\
  &\quad +\varep^{-\frac\theta{1-\theta}}(\bar c m)^\frac{\theta p}{1-\theta} \|u\|_{L^\alpha}^{\alpha+\mu} \Big(\int_U (1+|v|)^{\frac{\beta r_*}{1-r_*}}dx\Big)^\frac{(1-r_*)\theta}{r_*(1-\theta)}.
\end{aligned}
 \end{equation}

Now, for the last integral on the right hand side of \eqref{ay22}, using $v=u+\varphi$ and then by applying inequality \eqref{ee2} twice, we have
\begin{align*}
&\Big(\int_U (1+|v|)^{\frac{\beta r_*}{1-r_*}}dx\Big)^\frac{(1-r_*)\theta}{r_*(1-\theta)}
\le 2^{\frac{\beta\theta}{1-\theta}}
\Big(\int_U |u|^{\frac{\beta r_*}{1-r_*}}dx+\int_U (1+|\varphi|)^{\frac{\beta r_*}{1-r_*}}dx\Big)^\frac{(1-r_*)\theta}{r_*(1-\theta)}\\
&\le 2^{\frac{\beta\theta}{1-\theta}}\cdot 2^\frac{(1-r_*)\theta}{r_*(1-\theta)}
\left\{ \Big(\int_U |u|^{\frac{\beta r_*}{1-r_*}}dx\Big)^\frac{(1-r_*)\theta}{r_*(1-\theta)}+\Big(\int_U (1+|\varphi|)^{\frac{\beta r_*}{1-r_*}}dx\Big)^\frac{(1-r_*)\theta}{r_*(1-\theta)}\right\}.
\end{align*}
Then we obtain \eqref{ay0}.

\medskip
(ii) In case \eqref{alph2} is satisfied, by H\"older's inequality
\begin{align*}
\|u\|_{L^\frac{\beta r_*}{1-r_*}}^{\frac{\beta\theta}{1-\theta}} 
\le |U|^{\frac{\theta(\alpha(1-r_*)-\beta r_*)}{\alpha r_*(1-\theta)}} \|u\|_{L^\alpha}^\frac{\beta\theta}{1-\theta}.
\end{align*}
This and \eqref{ay0} yield inequality \eqref{ay}.
\end{proof}

Next, to carry out Moser's iterations in section \ref{maxsec} below, we need the following parabolic  multiplicative Sobolev inequality.
\begin{lemma}[Parabolic version]\label{PS3}
Assume \eqref{powcond} and 
\beq\label{alcond}
\alpha>0,\ 
\alpha\ge s,\ 
\alpha >\frac{(s-p)nr_*}{r_*(n+p)-n}.
\eeq 

Let $m$ and $\bar c$ be defined as in Lemma \ref{PS1}, and $T>0$.
Let $u(x,t)$ be function defined on $U\times(0,T)$ such that $|u(\cdot,t)|^m\in W^{1,r_*p}(U)$ and  $u(\cdot,t)$ vanishes on $\partial U$ for all $t\in (0,T)$. 

\begin{enumerate}[label=\tnum]
    \item\label{i26} Suppose $W(x,t)$ is a positive function  on $U\times (0,T)$.
Then one has
\beq\label{ppsi1}
\begin{aligned}
\|u\|_{L^{\kappa \alpha}(U\times(0,T))}
&\le (\bar c m)^\frac{p}{\kappa\alpha} \esssup_{t\in(0,T)}\|W^{-1}(\cdot,t)\|_{L^\frac{r_*}{1-r_*}}^\frac{1}{\kappa\alpha} \esssup_{t\in(0,T)}\|u(\cdot,t)\|_{L^\alpha}^{1-\tilde \theta}\\
&\quad\times \left(\int_0^T\int_U |u|^{\alpha-s}|\nabla u|^p Wdxdt\right)^\frac1{\kappa\alpha},
\end{aligned}
\eeq
where
\beq\label{defkappa}
\kappa= 1+\frac{r_*(n+p)-n}{nr_*}+\frac{p-s}{\alpha} >1, \quad 
\tilde \theta = \frac{1}{1+ \frac{\alpha(r_*(n+p)-n)}{nr_*(\alpha-s+p)} }\in (0,1).
\eeq 
    \item\label{ii26} Let $\varphi(x,t)$ be a function on $U\times (0,T)$, and define $v=u+\varphi$.
Then one has, for any $\beta>0$, that
\beq\label{ppsi3}
\begin{aligned}
&\|u\|_{L^{\kappa \alpha}(U\times(0,T))}
\le 2^{\frac1{\kappa\alpha}(\beta+\frac{1-r_*}{r_*})}(\bar c m)^\frac{p}{\kappa\alpha}\\
&\times \left(  \esssup_{t\in(0,T)}\|1+|\varphi(\cdot,t)|\|_{L^\frac{\beta r_*}{1-r_*}}^\beta  + \esssup_{t\in(0,T)}\|u(\cdot,t)\|_{L^\frac{\beta r_*}{1-r_*}}^\beta \right)^\frac1{\kappa\alpha}
\esssup_{t\in(0,T)}\|u(\cdot,t)\|_{L^\alpha}^{1-\tilde\theta} \\
&\times \left(\int_0^T\int_U |u|^{\alpha-s}|\nabla u|^p (1+|v|)^{-\beta}dx dt\right)^\frac1{\kappa\alpha}.
\end{aligned}
\eeq
\end{enumerate}
\end{lemma}
\begin{proof}
(i) Let  $\bar q=(r_*p)^*m$ as in \eqref{mqb}.
Denote $$I(t)=\int_U |u(x,t)|^{\alpha-s}|\nabla u(x,t)|^p W(x,t)dx.$$

Suppose, at the moment, $\kappa>1$ and  $\tilde\theta\in(0,1)$ are two numbers such that
\beq\label{tilthe} 
\frac{1}{\kappa\alpha}=\frac{\tilde \theta}{\bar q}+\frac{1-\tilde\theta}{\alpha}.
\eeq 

This particularly implies $\alpha<\kappa\alpha<\bar q$.
Applying  the interpolation inequality gives
\beqs
 \|u(\cdot,t)\|_{L^{\kappa \alpha}}  \le \|u(\cdot,t)\|_{L^{\bar q}}^{\tilde \theta}\|u(\cdot,t)\|_{L^\alpha}^{1-\tilde\theta}.
\eeqs

Applying inequality \eqref{229a}  to estimate $\|u(\cdot,t)\|_{L^{\bar q}}$ on the right-hand side,  and raising both sides of the resulting inequality to power $\kappa \alpha$ yield  
\beq \label{L24}
\begin{aligned}
 \int_U |u(x,t)|^{\kappa \alpha}dx 
&\le \left [(\bar c m)^\frac1m I(t)^\frac1{mp} \|W^{-1}(\cdot,t)\|_{L^\frac{r_*}{1-r_*}}^\frac1{mp}\right]^{\tilde\theta \kappa\alpha} 
\|u\|_{L^\alpha}^{(1-\tilde\theta)\kappa\alpha}.
 \end{aligned}
 \eeq 
 
We impose the condition
 \beq\label{mptil}
 mp=\tilde\theta\kappa\alpha.
 \eeq

Then \eqref{L24} becomes
\beq\label{L232}
 \int_U |u(x,t)|^{\kappa \alpha}dx 
\le (\bar c m)^p I(t) \|W^{-1}(\cdot,t)\|_{L^\frac{r_*}{1-r_*}} \|u(\cdot,t)\|_{L^\alpha}^{(1-\tilde\theta)\kappa\alpha}.
\eeq
 
 Integrating \eqref{L232} in $t$ from $0$ to $T$, we have
\begin{align*}
 \int_0^T\int_U |u(x,t)|^{\kappa \alpha}dx dt
\le (\bar c m)^p \esssup_{t\in(0,T)}\|W^{-1}(\cdot,t)\|_{L^\frac{r_*}{1-r_*}}
\esssup_{t\in(0,T)}\|u(\cdot,t)\|_{L^\alpha}^{(1-\tilde\theta)\kappa\alpha}\int_0^T I(t)dt.
 \end{align*}
Taking power $1/\kappa\alpha$ of both sides of the last inequality gives \eqref{ppsi1}. 
 
 It remains to verify \eqref{tilthe}  and \eqref{mptil}. We compute $\kappa$ and $\tilde \theta$ explicitly from these two equations.
 Multiplying \eqref{tilthe} by $\kappa\alpha$ and then using relation \eqref{mptil}, we have 
\beqs 
1=\frac{mp}{m(r_*p)^*}+\kappa -\frac{mp}{\alpha}=\frac{n-r_*p}{nr_*}+\kappa -\frac{mp}{\alpha}.
\eeqs 

Solving for $\kappa$ and using the value of $m$ given in \eqref{mdef}, we have
\beq\label{kk}
\kappa=1 - \frac{n-r_*p}{nr_*} + \frac{mp}{\alpha}
=\frac{r_*(n+p)-n}{nr_*} + \frac{\alpha-s+p}{\alpha}.
\eeq
Then formula \eqref{defkappa} of $\kappa$ follows.
Using formula of $\kappa$ in \eqref{kk}, and again, formula \eqref{mdef} for $m$, we calculate $\tilde\theta$ from \eqref{mptil} by
\beqs
\tilde\theta=\frac{mp}{\kappa\alpha} = \frac{\alpha-s+p}{ \frac{\alpha (r_*(n+p)-n)}{nr_*}+\alpha-s+p}
=\frac{1}{\frac{\alpha(r_*(n+p)-n)}{nr_*(\alpha-s+p)} +1}.
\eeqs 
We then obtain the formula of $\tilde\theta$ in \eqref{defkappa}.
Because $\alpha>0$, $\alpha\ge s$ and $r_*>n/(n+p)$, we have $\tilde \theta\in(0,1)$.
By the last condition in \eqref{alcond}, we have $\kappa>1$. 
The proof  of \eqref{ppsi1} is complete.

\medskip
(ii) Applying inequality \eqref{ppsi1} to $W(x,t)=(1+|v|)^{-\beta}$ with $v=u+\varphi$, we have
 \beq\label{ppsi2}
\begin{aligned}
\|u\|_{L^{\kappa \alpha}(U\times(0,T))}
&\le (\bar c m)^\frac{p}{\kappa\alpha} \esssup_{t\in(0,T)}\left(\int_U (1+|v|)^\frac{\beta r_*}{1-r_*}dx\right)^\frac{1-r_*}{r_*\kappa\alpha} 
\esssup_{t\in(0,T)}\|u(\cdot,t)\|_{L^\alpha}^{1-\tilde \theta}\\
&\quad\times \left(\int_0^T\int_U |u|^{\alpha-s}|\nabla w|^p (1+|v|)^{-\beta} dxdt\right)^\frac1{\kappa\alpha}.
\end{aligned}
\eeq

By triangle inequality and \eqref{ee2}, we have 
\beq\label{L234}
\begin{aligned}
&\left(\int_U (1+|v|)^\frac{\beta r_*}{1-r_*}dx\right)^\frac{1-r_*}{r_*\kappa\alpha} 
\le 2^\frac{\beta}{\kappa\alpha}\left(\int_U (1+|\varphi|)^\frac{\beta r_*}{1-r_*}dx+\int_U |u|^\frac{\beta r_*}{1-r_*}dx\right)^\frac{1-r_*}{r_*\kappa\alpha} \\
&\le 2^\frac{\beta}{\kappa\alpha}2^\frac{1-r_*}{r_*\kappa\alpha}
\left\{ \left(\int_U (1+|\varphi|)^\frac{\beta r_*}{1-r_*}dx\right)^\frac{1-r_*}{r_*}
+\left(\int_U |u|^\frac{\beta r_*}{1-r_*}dx\right)^\frac{1-r_*}{r_*} \right\}^\frac1{\kappa\alpha}.
\end{aligned}
\eeq
Combining \eqref{ppsi2} and \eqref{L234} yields  \eqref{ppsi3}.
\end{proof}

\section{Estimates for the Lebesgue norms}\label{L-p-est}

Let $u$ be a  nonnegative solution of problem \eqref{ibvpg} in a domain $U\subset \R^3$. 
We will derive estimates for the $L^\alpha$-norms of  $u$ for $\alpha>0$.  
To do so, it is convenient to shift $u$ by its boundary values and deal with a function vanishing on the boundary.

Let $\Psi(x,t)$ be an extension of the boundary data $\psi(x,t)$ from $\Gamma\times(0,T]$ to $\bar U\times[0,T]$. 

Define $\bar u(x,t)=u(x,t)-\Psi(x,t)$ and $\bar u_0(x)=u_0(x)-\Psi(x,0)$.
We derive from \eqref{ibvpg} the equations for $\bar u$:
\beq\label{ubar0}
\begin{aligned}
\begin{cases}
 \begin{displaystyle}\frac{\partial \bar u}{\partial t}\end{displaystyle}=\nabla\cdot (X(u, \Phi(x,t))- \Psi_t\quad &\text{on}\quad U\times (0,\infty),\\
\bar u(x,0)=\bar u_0(x) \quad &\text{on}\quad U,\\
\bar u(x,t)= 0\quad &\text{on}\quad \Gamma\times (0,\infty),
\end{cases}
\end{aligned}
\eeq
where $\Psi_t=\partial\Psi/\partial t$ and
\beq\label{Phi00}
\Phi(x,t)=\nabla u(x,t)+u^2(x,t)\mathcal Z(x,t).
\eeq

We will focus on estimating solution $\bar u$ of \eqref{ubar0}.
Clearly, corresponding estimates for $u=\bar u+\Psi$ will easily follow.

By definition \eqref{Zx1} of $\mathcal Z(\cdot,t)$ and \eqref{Jineq}, we have 
\beq\label{ZZ}
|\mathcal Z(x,t)| \le \mathcal{G}+\Omega^2|\mathbf J^2x|\le M_\mathcal Z\eqdef  \mathcal{G}+\Omega^2 r_U
=\mathcal G+\frac{\phi^2 r_U}{4\varpi^2}R_*^2,
\eeq
where $r_U=\max\{ |x|: x\in\bar{U} \}$. 
Set 
\beq\label{chidef}
\chi_*=\max\{1,R_*^2,\chi_0^2,M_{\mathcal{Z}}\}.
\eeq

We will use $\chi_*$ as the main parameter to measure the effect of the rotation.
Our estimates in this paper will be expressed in terms of $\chi_*$.

It follows \eqref{ZZ} and \eqref{chidef} that
\beq\label{Mzbnd}
|\mathcal Z(x,t)| \le \chi_* \text{ for all $x\in U$ and $t>0$}.
\eeq

We start estimating the Lebesgue norms of the solutions with the following differential inequality.

\begin{lemma}\label{diff1} Assume 
\beq \label{rcond}
\frac{3}{5-a}<r_*<1,\quad r_*\ge \frac1{2-a},
\eeq 
and 
\beq\label{alph1} 
\alpha\ge \frac{2r_*}{1-r_*}, \quad 
\alpha> \frac{3r_*(4-3a)}{r_*(5-a)-3}.
\eeq 
Then one has
\beq\label{al1}
\begin{aligned}
&\frac{d}{dt}\int_U |\bar{u}|^\alpha dx+\frac{c_4}{8 }\chi_*^{-1}\int_U (1+u)^{-2}   |\nabla \bar u|^{2-a} |\bar{u}|^{\alpha-2}dx \\
&\le   C_0\chi_*^{\bar\mu } \left\{\left( \int_U |\bar u|^\alpha dx \right)^{1+\gamma} 
+\left(\int_U |\bar u|^\alpha dx \right)^{1+\gamma'}\|1+|\Psi|\|_{L^\frac{2r_*}{1-r_*}}^\frac{2\theta}{1-\theta}\right\} +C_0 M(t),
\end{aligned}
\eeq
where 
\beq\label{ggam}
\gamma=\frac{\mu}{\alpha}+\frac{2\theta}{\alpha(1-\theta)}>0,\quad 
\gamma'=\frac{\mu}\alpha>0,\quad 
\bar\mu=\frac{3-2a+\theta}{1-\theta}>0
\eeq 
with 
\beq\label{mt2}
\theta =\frac{12(1-a)r_*}{\alpha((5-a)r_*-3)-3ar_*}\in(0,1),\quad 
\mu =\frac{4(1-a)+a\theta}{1-\theta}>0,
\eeq 
positive constants $C_0$ is defined in \eqref{CC} below, which depends on $\alpha$, and
\beq\label{Mdef}
\begin{aligned}
M(t)&=\chi_*^{3-2a}\int_U (1+ |\Psi(x,t)|)^{\alpha+4(1-a)} dx
 +\chi_*^{1-a}\int_U (1+|\nabla\Psi(x,t)|)^\frac{\alpha+4(1-a)}{2}dx \\
&\quad + \chi_*^{-\frac{(3-2a)(\alpha-1)}{5-4a}} \int_U |\Psi_t(x,t)|^\frac{\alpha+4(1-a)}{5-4a}dx.
\end{aligned}
\eeq
\end{lemma}
\begin{proof} We proceed in two steps.
In the calculations below, the constants $c_i$'s are independent of $\alpha$ and $\chi_*$, while $C_j$'s are independent of $\chi_*$, but dependent on $\alpha$.

\medskip
\emph{Step 1.} We derive a weaker version of \eqref{al1} first, namely, inequality \eqref{diu} below.

Since $r_*\ge 1/(2-a)>1/2$, one has 
\beq\label{a2} 
\alpha\ge \frac{2r_*}{1-r_*}> \frac{2(1/2)}{1-1/2}=2.
\eeq 
Then we can use the following identities
\beqs 
\frac{\partial}{\partial t} (|\bar u|^\alpha) =\alpha|\bar u|^{\alpha-2}\bar u \frac{\partial \bar u}{\partial t},\quad 
\nabla (|\bar{u}|^{\alpha-2}\bar u)= (\alpha-1)|\bar u|^{\alpha-2}\nabla \bar u.
\eeqs 

Multiplying the PDE in \eqref{ubar0} by $|\bar{u}|^{\alpha-2}\bar u$, integrating over domain $U$, and using integration by parts, we have
\begin{align*}
\frac{1}{\alpha}\frac{d}{dt}\int_U |\bar{u}|^\alpha dx
&  = -(\alpha-1)\int_U X(u, \Phi(x,t)) \cdot(|\bar{u}|^{\alpha-2}\nabla \bar{u})dx-   \int_U\Psi_t|\bar{u}|^{\alpha-2}\bar{u}dx.
\end{align*} 

On the right-hand side, we write $\nabla\bar{u}=\nabla u-\nabla \Psi$ and use  relation \eqref{Phi00} for $\nabla u$ to have 
$$\nabla \bar u(x,t)=\Phi(x,t)- u^2(x,t)\mathcal Z(x,t)-\nabla\Psi=\Phi(x,t)- (\bar u+\Psi)^2\mathcal Z(x,t)-\nabla \Psi.$$ 
We obtain
\beq \label{du1}
\begin{aligned}
\frac{1}{\alpha}\frac{d}{dt}\int_U |\bar{u}|^\alpha dx
&  =-(\alpha-1)\int_U X(u, \Phi(x,t))\cdot \Phi |\bar{u}|^{\alpha-2}dx \\
&\quad +(\alpha-1)\int_U X(u, \Phi(x,t))\cdot \mathcal Z |\bar{u}|^{\alpha-2}(\bar u+\Psi)^2 dx\\
&\quad +(\alpha-1)\int_U X(u, \Phi(x,t))\cdot\nabla\Psi |\bar{u}|^{\alpha-2}dx- \int_U\Psi_t|\bar{u}|^{\alpha-2}\bar u dx.
\end{aligned} 
\eeq 

For the first integral on the right-hand side of \eqref{du1}, we use the first inequality in \eqref{X3}  to estimate 
\begin{align*}
X(u, \Phi)\cdot \Phi&\ge c_5(\chi_0+R_*u)^{-2} (|\Phi|^{2-a}-1)\\
&\ge c_5 \max\{\chi_0,R_*\}^{-2}(1+u)^{-2}|\Phi|^{2-a}
-c_5 \chi_0^{-2}\\
&\ge c_5 \chi_*^{-1}(1+u)^{-2}|\Phi|^{2-a}
-c_5 \chi_0^{-2}.
\end{align*}

For the second integral on the right-hand side of \eqref{du1}, we use Cauchy-Schwarz inequality, the second inequality in \eqref{X1} and \eqref{Mzbnd} to have 
\begin{align*}
|X(u, \Phi)\cdot \mathcal Z|(\bar u+\Psi)^2
&\le  |X(u,\Phi )| |\mathcal Z| \cdot 2(\bar u^2+\Psi^2)
\le 2c_3|\Phi|^{1-a}\chi_* (\bar u^2+\Psi^2).
\end{align*}

By the triangle inequality,  \eqref{Mzbnd} and then \eqref{ee3},
\beqs
|\Phi|^{1-a}\le (|\nabla u|+u^2 \chi_*)^{1-a}\le |\nabla u|^{1-a}+|u|^{2(1-a)}\chi_*^{1-a}.
\eeqs 

Therefore,
\begin{align*}
|X(u, \Phi)\cdot \mathcal Z|(\bar u+\Psi)^2
&\le 2c_3 \chi_*(|\nabla u|^{1-a}+\chi_*^{1-a}|u|^{2(1-a)})(\bar u^2+\Psi^2).
\end{align*}

Similarly, for the third integral on the right-hand side of \eqref{du1},
\begin{align*}
&|X(u, \Phi)\cdot\nabla\Psi |
\le |X(u, \Phi)| |\nabla\Psi |\le  c_3 |\Phi|^{1-a} |\nabla\Psi |
\le c_3 (|\nabla u|^{1-a}+\chi_*^{1-a}|u|^{2(1-a)})|\nabla\Psi |.
\end{align*}

Combining the above estimates of the terms in \eqref{du1} gives
\beq\label{di1}
\frac{1}{\alpha}\frac{d}{dt}\int_U |\bar{u}|^\alpha dx
  \le (\alpha-1)( -I_0+I_1+I_2+I_3+I_4)+I_5,
\eeq 
where
\begin{align*}
&I_0=c_5\chi_*^{-1}\int_U (1+u)^{-2} |\Phi|^{2-a} |\bar{u}|^{\alpha-2}dx,
&&I_1=c_9\int_U |\bar{u}|^{\alpha-2}dx ,\\
&I_2=2 c_3\int_U  (\chi_*|\nabla u|^{1-a}+\chi_*^{2-a}|u|^{2(1-a)})|\bar{u}|^{\alpha}dx,\\
&I_3=2c_3\int_U (\chi_*|\nabla u|^{1-a}+\chi_*^{2-a}|u|^{2(1-a)})|\bar{u}|^{\alpha-2}\Psi^2dx\\
&I_4=c_3\int_U (|\nabla u|^{1-a}+\chi_*^{1-a}|u|^{2(1-a)}) |\nabla\Psi|\, |\bar{u}|^{\alpha-2}dx,
&&I_5= \int_U|\bar{u}|^{\alpha-1}|\Psi_t|dx
\end{align*} 
with $c_9=c_5 \chi_0^{-2}$.

\medskip
\noindent\textit{Estimation of $I_0$.} Applying inequality \eqref{ee6} and estimate \eqref{Mzbnd}, we have
\begin{align*}
|\Phi|^{2-a}
&=|\nabla u+u^2\mathcal Z|^{2-a}\ge 2^{a-1}|\nabla u|^{2-a}-|u|^{2(2-a)}|\mathcal Z|^{2-a}\\
&\ge 2^{a-1}|\nabla u|^{2-a}-|u|^{2(2-a)}\chi_*^{2-a}.
\end{align*}

Hence,
\begin{align*}
    -I_0&\le -2^{a-1}c_5\chi_*^{-1}\int_U (1+u)^{-2} |\nabla u|^{2-a} |\bar{u}|^{\alpha-2}dx
    +c_5\chi_*^{1-a}\int_U (1+u)^{-2}|u|^{2(2-a)}|\bar{u}|^{\alpha-2}dx.
\end{align*}

Regarding the last integrand, we note, for $\beta>0$ and $\gamma\ge 0$, that 
\beq\label{ubg}
u^\beta|\bar{u}|^\gamma
\le (|\bar u|+|\Psi|)^\beta|\bar{u}|^\gamma
\le 2^\beta(|\bar u|^\beta+|\Psi|^\beta)|\bar{u}|^\gamma\\
=2^\beta(|\bar u|^{\beta+\gamma}+|\bar{u}|^\gamma |\Psi|^\beta).
\eeq 

Use $(1+u)^{-2}\le u^{-2}$ and \eqref{ubg} with $\beta=2-2a$ and $\gamma=\alpha-2$, we have 
\beqs 
(1+u)^{-2}|u|^{2(2-a)}|\bar{u}|^{\alpha-2}
\le |u|^{2-2a}|\bar{u}|^{\alpha-2}\le 2^{2-2a}(|\bar u|^{\alpha-2a}+|\Psi|^{2-2a}|\bar{u}|^{\alpha-2}).
\eeqs 

Therefore,
\beqs
-I_0 
\le -2^{-1}c_4\chi_*^{-1}\int_U (1+u)^{-2} |\nabla u|^{2-a}|\bar{u}|^{\alpha-2}dx+J_0,
\eeqs
where
\beqs
J_0=4^{1-a}c_5 \chi_*^{1-a}\left\{\int_U  |\bar{u}|^{\alpha-2a}dx+\int_U  |\bar{u}|^{\alpha-2}|\Psi|^{2-2a} dx\right\}.
\eeqs

\medskip
\noindent\textit{Estimation of $I_2$.} 
We will use the following estimates.

Let $P\ge 0$ and $\varep>0$. We write
\begin{align*}
|\nabla u|^{1-a}P
&=\left(\varep^{(1-a)/(2-a)} |\nabla u|^{1-a} (1+u)^{-2(1-a)/(2-a)} \right)\times \left(\varep^{-(1-a)/(2-a)}(1+u)^{2(1-a)/(2-a)}P\right).
\end{align*}

Applying Young's inequality \eqref{eeY} with powers $(2-a)/(1-a)$ and $2-a$ to the last product, and multiplying the resulting inequality by $|\bar{u}|^{\alpha-2}$, we obtain 
\beqs 
|\nabla u|^{1-a}|\bar{u}|^{\alpha-2}P
\le \varep |\nabla u|^{2-a}(1+u)^{-2}|\bar u|^{\alpha-2}+\varep^{-(1-a)}(1+u)^{2(1-a)}|\bar{u}|^{\alpha-2}P^{2-a}.
\eeqs

For the last term, writing $1+u=\bar{u}+1+\Psi$, then using the triangle inequality and \eqref{ee2}, we have
\begin{align*}
(1+u)^{2(1-a)}\le (|\bar{u}|+1+|\Psi|)^{2(1-a)}\le 2^{2(1-a)}(|\bar u|^{2(1-a)}+(1+|\Psi|)^{2(1-a)}).
\end{align*}

Thus,
\beq \label{uP}
\begin{aligned}
|\nabla u|^{1-a}|\bar{u}|^{\alpha-2}P
&\le \varep |\nabla u|^{2-a}(1+u)^{-2}|\bar u|^{\alpha-2}\\
&\quad +4^{1-a}\varep^{-(1-a)}\left( |\bar u|^{\alpha-2a}P^{2-a}+|\bar{u}|^{\alpha-2}(1+|\Psi|)^{2(1-a)}P^{2-a}
\right).
\end{aligned}
\eeq 

Letting $P=|\bar u|^2$ in \eqref{uP}, we have 
\beq\label{u0}
\begin{aligned}
|\nabla u|^{1-a}|\bar{u}|^\alpha
&\le \varep |\nabla u|^{2-a}(1+u)^{-2}|\bar u|^{\alpha-2}\\
&\quad +4^{1-a}\varep^{-(1-a)}(|\bar u|^{\alpha+4(1-a)}+|\bar{u}|^{\alpha+2(1-a)}(1+|\Psi|)^{2(1-a)}).
\end{aligned}
\eeq 
For the rest of $I_2$, letting $\beta=2-2a$ and $\gamma=\alpha$ in \eqref{ubg}, we have
\beq \label{uu}
u^{2-2a}|\bar{u}|^{\alpha}
\le 2^{2-2a}(|\bar u|^{\alpha+2-2a}+|\Psi|^{2-2a}|\bar{u}|^{\alpha}).
\eeq 

Therefore, \eqref{u0} and \eqref{uu} yield
\begin{align*}
I_2&=2c_3\int_U  (\chi_*|\nabla u|^{1-a}|\bar{u}|^{\alpha}+\chi_*^{2-a}|u|^{2(1-a)}|\bar{u}|^{\alpha})dx\\
&\le 2\varep c_3 \chi_*\int_U (1+u)^{-2}|\nabla u|^{2-a} |\bar{u}|^{\alpha-2}dx+J_2,
\end{align*}
where
\begin{align*}
J_2
&=2\cdot 4^{1-a}\varep^{-(1-a)}c_3\chi_*\left\{ \int_U |\bar{u}|^{\alpha+4(1-a)}dx
 +\int_U |\bar{u}|^{\alpha+2(1-a)} (1+|\Psi|)^{2(1-a)}dx \right\}\\
&\quad +2\cdot 4^{1-a} c_3\chi_*^{2-a}\int_U (|\bar u|^{\alpha+2-2a} + |\bar{u}|^{\alpha}|\Psi|^{2-2a})dx.
\end{align*}

\medskip
\noindent\textit{Estimation of $I_3$.} Using \eqref{uP} with $P=\Psi^2$, and applying \eqref{ubg} to $\beta=2(1-a)$ and $\gamma=\alpha-2$, we obtain
\begin{align*}
I_3&=  2c_3\int_U (\chi_*|\nabla u|^{1-a} |\bar{u}|^{\alpha-2}\Psi^2 + \chi_*^{2-a}|u|^{2-2a}|\bar{u}|^{\alpha-2}\Psi^2) dx\\
&\le 2\varep c_3\chi_*\int_U (1+u)^{-2}|\nabla u|^{2-a} |\bar{u}|^{\alpha-2}dx +J_3,
\end{align*}
where
\begin{align*}
J_3&=2\cdot 4^{1-a}\varep^{-(1-a)}c_3\chi_*\left\{\int_U |\bar{u}|^{\alpha-2a}|\Psi|^{2(2-a)} dx+\int_U |\bar{u}|^{\alpha-2}(|1+|\Psi|)^{2(1-a)}|\Psi|^{2(2-a)}dx \right\}\\
&\quad +2\cdot 4^{1-a}c_3\chi_*^{2-a}\left\{\int_U |\bar{u}|^{\alpha-2a}\Psi^2dx
+\int_U |\bar{u}|^{\alpha-2}|\Psi|^{4-2a} dx\right\}.
\end{align*}

\medskip
\noindent\textit{Estimation of $I_4$.} 
Let $\delta>0$. 
Using \eqref{uP} again with $\varep=\delta$ and $P=|\nabla \Psi|$, and applying inequality \eqref{ubg} to $\beta=2-2a$ and $\gamma=\alpha-2$, we have 
\begin{align*}
I_4
&\le \delta c_3\int_U (1+u)^{-2}|\nabla u|^{2-a}|\bar{u}|^{\alpha-2}dx+J_4,
\end{align*}
where
\begin{align*}
J_4
&=4^{1-a}c_3 \delta^{-(1-a)}\left\{\int_U |\bar{u}|^{\alpha-2a}|\nabla\Psi|^{2-a} dx
+\int_U |\bar{u}|^{\alpha-2}(1+|\Psi|)^{2(1-a)}|\nabla\Psi|^{2-a}dx\right\}\\
&\quad +4^{1-a}c_3\chi_*^{1-a}\left\{\int_U |\bar{u}|^{\alpha-2a}|\nabla\Psi| dx
+\int_U |\bar{u}|^{\alpha-2}|\Psi|^{2-2a}|\nabla\Psi| dx\right\}.
\end{align*}

Combining \eqref{di1} with the above estimates of $I_i$, for $i=0,2,3,4$, we have 
\beq\label{ddtpf}
\begin{aligned}
\frac{1}{\alpha}\frac{d}{dt}\int_U |\bar{u}|^\alpha dx
&\le -(\alpha-1) C_{\varep,\delta}\int_U (1+u)^{-2} |\nabla u|^{2-a} |\bar{u}|^{\alpha-2}dx\\
&\quad + (\alpha-1)(J_0+I_1+J_2+J_3+J_4)+I_5,
\end{aligned}
\eeq 
where 
$C_{\varep,\delta}=2^{-1}c_4\chi_*^{-1} -4 c_3 \varep \chi_* - \delta c_3$.
We select
\beq \label{ed}
\varep=\frac{c_4}{32c_3}\chi_*^{-2}\text{ and } \delta=\frac{c_4}{8c_3}\chi_*^{-1},
\text{ then $C_{\varep,\delta}=c_{10}\chi_*^{-1}$ with $c_{10} =c_4/4$.}
\eeq 

We estimate the first integral on the right-hand side of \eqref{ddtpf}, using \eqref{ee6} and the fact $1+u\ge 1$, by
\begin{align*}
(1+u)^{-2} |\nabla u|^{2-a} |\bar{u}|^{\alpha-2}
&=(1+u)^{-2} |\nabla\bar u+\nabla\Psi|^{2-a} |\bar{u}|^{\alpha-2}\\
&\ge (1+u)^{-2} \left(2^{a-1} |\nabla \bar u|^{2-a}-|\nabla\Psi|^{2-a}\right)|\bar{u}|^{\alpha-2}\\
&\ge 2^{-1} (1+u)^{-2}  |\nabla \bar u|^{2-a}|\bar{u}|^{\alpha-2}-|\nabla\Psi|^{2-a}|\bar{u}|^{\alpha-2}.
\end{align*}
Then
\beq\label{ddtpf1}
\begin{aligned}
\frac{1}{\alpha}\frac{d}{dt}\int_U |\bar{u}|^\alpha dx
&\le -\frac{c_{10}(\alpha-1)}{2\chi_*}\int_U (1+u)^{-2} |\nabla \bar u|^{2-a} |\bar{u}|^{\alpha-2}dx \\
&\quad + (\alpha-1)(J_0+I_1+J_2+J_3+J_4+I_6)+I_5,
\end{aligned}
\eeq 
where 
\beqs
   I_6=c_{10}\chi_*^{-1} \int_U  |\bar{u}|^{\alpha-2} |\nabla \Psi|^{2-a} dx. 
\eeqs

We continue to estimate the right-hand side of \eqref{ddtpf1}.

\medskip
\noindent\textit{Estimation of $J_0,I_1,J_4,I_6$.}
We combine $J_0$, $I_1$, $J_4$, with $\delta$ in \eqref{ed}, and $I_6$, and collect the corresponding terms of $|\bar{u}|^{\alpha-2a}$ and $|\bar{u}|^{\alpha-2}$. All together, they result in
\beq\label{JIs}
J_0+I_1+J_4+I_6
\le c_{11}\chi_*^{1-a}\int_U( |\bar{u}|^{\alpha-2a}(1+|\nabla\Psi|)^{2-a} +|\bar{u}|^{\alpha-2}(1+|\Psi|)^{2(1-a)}(1+|\nabla\Psi|)^{2-a})dx,
\eeq 
where
$c_{11}=4^{1-a}c_5+c_9+2c_3(32c_3/c_4)^{1-a}+c_{10} $.
In the last integral, applying Young's inequality \eqref{eeY} with 
\begin{align*}
&\text{powers $\frac{\alpha+4-4a}{\alpha-2a}$ and $\frac{\alpha+4-4a}{4-2a}$ to the product $|\bar{u}|^{\alpha-2a}(1+|\nabla\Psi|)^{2-a}$, } \text{ and }\\
&\text{powers $\frac{\alpha+4-4a}{\alpha-2}$, $\frac{\alpha+4-4a}{2-2a}$, $\frac{\alpha+4-4a}{4-2a}$ to the product} \\
& \text{$|\bar{u}|^{\alpha-2} (1+|\Psi|)^{2(1-a)} (1+|\nabla\Psi|)^{2-a}$,   we obtain}  
\end{align*}
\beqs
J_0+I_1+J_4+I_6 
\le c_{11}\chi_*^{1-a}\int_U \Big\{ 2|\bar{u}|^{\alpha+4(1-a)}+(1+|\Psi|)^{\alpha+4(1-a)}+2(1+|\nabla\Psi|)^\frac{\alpha+4(1-a)}{2}\Big\}dx.
\eeqs

We conveniently split the last integral with the use of the fact $\chi_*^{1-a}\le \chi_*^{3-2a}$ for the first two terms, and obtain
\beq\label{JIJI}
\begin{aligned}
J_0+I_1+J_4+I_6 
&\le c_{11}\chi_*^{3-2a}\int_U (2|\bar{u}|^{\alpha+4(1-a)}+(1+|\Psi|)^{\alpha+4(1-a)})dx\\
&\quad + 2c_{11}\chi_*^{1-a}\int_U (1+|\nabla\Psi|)^\frac{\alpha+4(1-a)}{2}dx.
\end{aligned}
\eeq 

\medskip
\noindent\textit{Estimation of $J_2,J_3$.}
With the value of $\varep$ in \eqref{ed}, we evaluate and estimate $J_2$ and $J_3$. Note that the powers for $\chi_*$ in $J_2$ and $J_3$ are $3-2a$ and $2-a$ now. Since $3-2a>2-a$ and $\chi_*\ge 1$, we estimate
\beq \label{pJ2} 
J_2
\le c_{12}\chi_*^{3-2a}\int_U \left\{  |\bar{u}|^{\alpha+4(1-a)} 
+  |\bar{u}|^{\alpha+2-2a}(1+|\Psi|)^{2(1-a)}+   |\bar{u}|^{\alpha}|\Psi|^{2-2a}\right\}dx, 
\eeq 
where $c_{12}= 2\cdot 4^{1-a}c_3[1+(32c_3/c_4)^{1-a}]$. 
Similarly,
\beq\label{pJ3}
J_3
\le c_{12}\chi_*^{3-2a}\int_U \left\{  |\bar{u}|^{\alpha-2a}(1+|\Psi|)^{2(2-a)}+  |\bar{u}|^{\alpha-2}(1+|\Psi|)^{2(3-2a)}
\right\}dx.
\eeq 

For the  integrals on the right-hand side \eqref{pJ2} and \eqref{pJ3}, by applying Young's inequality \eqref{eeY} with 
\begin{align*}
&\text{powers $\frac{\alpha+4-4a}{\alpha+2-2a}$ and $\frac{\alpha+4-4a}{2-2a}$ to the product 
$|\bar{u}|^{\alpha+2-2a}(1+|\Psi|)^{2(1-a)}$,}\\
&\text{powers $\frac{\alpha+4-4a}{\alpha}$ and $\frac{\alpha+4-4a}{4-4a}$ to the product
$|\bar{u}|^{\alpha}|\Psi|^{2-2a}$,}\\
&\text{powers $\frac{\alpha+4-4a}{\alpha-2a}$ and $\frac{\alpha+4-4a}{4-2a}$  to the product 
$|\bar{u}|^{\alpha-2a}(1+|\Psi|)^{2(2-a)}$,}\\
&\text{powers $\frac{\alpha+4-4a}{\alpha-2}$ and $\frac{\alpha+4-4a}{6-4a}$ to the product 
$ |\bar{u}|^{\alpha-2}(1+|\Psi|)^{2(3-2a)}$,}
\end{align*}
 we obtain
\begin{align}
J_2+J_3
&\le c_{12}\chi_*^{3-2a}\int_U  (5|\bar{u}|^{\alpha+4(1-a)} +3(1+|\Psi|)^{\alpha+4(1-a)}
+|\Psi|^\frac{\alpha+4(1-a)}{2})dx \notag \\
&\le c_{12}\chi_*^{3-2a}\int_U  (5|\bar{u}|^{\alpha+4(1-a)} +4(1+|\Psi|)^{\alpha+4(1-a)})dx.\label{J23}
\end{align}

\medskip
\noindent\textit{Estimation of $I_5$.} We write
\beqs
 |\bar u|^{\alpha-1} |\Psi_t|=   \Big(\chi_*^\frac{(3-2a)(\alpha-1)}{\alpha+4-4a} |\bar u|^{\alpha-1} \Big) \cdot \Big( \chi_*^{-\frac{(3-2a)(\alpha-1)}{\alpha+4-4a}} |\Psi_t|\Big).
\eeqs

Applying Young's inequality \eqref{eeY} with powers $(\alpha+4-4a)/(\alpha-1)$ and $(\alpha+4-4a)/(5-4a)$  to the last product yields
\begin{align}\label{I5}
I_5&\le \chi_*^{3-2a}\int_U |\bar{u}|^{\alpha+4(1-a)}dx
+\chi_*^{-\frac{(3-2a)(\alpha-1)}{5-4a}} \int_U |\Psi_t|^\frac{\alpha+4(1-a)}{5-4a}dx \notag\\
&\le \chi_*^{3-2a}(\alpha-1)\int_U |\bar{u}|^{\alpha+4(1-a)}dx
+\chi_*^{-\frac{(3-2a)(\alpha-1)}{5-4a}} \int_U |\Psi_t|^\frac{\alpha+4(1-a)}{5-4a}dx.
\end{align}

\medskip
Combining \eqref{ddtpf1} with  estimates \eqref{JIJI}, \eqref{J23},  and \eqref{I5} gives
\beq\label{diu}
\begin{aligned}
&\frac{1}{\alpha}\frac{d}{dt}\int_U |\bar{u}|^\alpha dx
\le -\frac{c_{10}(\alpha-1)}{2\chi_*}\int_U (1+u)^{-2} |\nabla \bar u|^{2-a} |\bar{u}|^{\alpha-2}dx \\
&\quad + c_{13}\chi_*^{3-2a}(\alpha-1)\int_U  |\bar{u}|^{\alpha+4(1-a)} dx
+c_{14}\chi_*^{3-2a}(\alpha-1)\int_U  (1+|\Psi|)^{\alpha+4(1-a)} dx\\
&\quad + 2c_{11}\chi_*^{1-a}(\alpha-1)\int_U (1+|\nabla\Psi|)^\frac{\alpha+4(1-a)}{2}dx 
+\chi_*^{-\frac{(3-2a)(\alpha-1)}{5-4a}} \int_U |\Psi_t|^\frac{\alpha+4(1-a)}{5-4a}dx,
\end{aligned}
\eeq 
where
$c_{13}= 5c_{12}+2c_{11}+1$ and $c_{14}=4c_{12}+c_{11}$.

\medskip
\emph{Step 2.} We improve inequality \eqref{diu} to obtain \eqref{al1}.

We apply Corollary \ref{PS2}\ref{ii24} to $n=3$, $p=2-a$, $r=4(1-a)$, $s=2$,  $\beta=2$, and functions $u:=\bar u(\cdot,t)$, $\varphi:=\Psi(\cdot,t)$, $v:=\bar u(\cdot,t)+\Psi(\cdot,t)=u(\cdot,t)\ge 0$.
Note, in this case, that the number $m$ in \eqref{mdef} is $m=(\alpha-a)/(2-a)$.

Because $n/p=3/(2-a)>1$, then condition \eqref{powcond} becomes \eqref{rcond}.
Also, because of \eqref{a2}, the conditions on $\alpha$ in \eqref{alz} and condition \eqref{alph2} become \eqref{alph1}.
Then it follows inequality \eqref{ay}, for any $\varep>0$, that
\beq \label{imb}
\begin{aligned}
 \int_U |\bar u|^{\alpha+4(1-a)} dx
  &\le \varep \int_U (1+u)^{-2}|\nabla \bar u|^{2-a} |\bar u|^{\alpha-2}dx\\
&\quad  +C_1 \varep^{-\frac\theta{1-\theta}}  \Big(\|\bar u\|_{L^\alpha}^{\alpha+\mu+\frac{2\theta}{1-\theta}}+
  \|\bar u\|_{L^\alpha}^{\alpha+\mu} \|1+|\Psi|\|_{L^\frac{2r_*}{1-r_*}}^\frac{2\theta}{1-\theta}\Big),
  \end{aligned}
  \eeq 
where
\beq\label{C1} 
C_1=\left[ 2^\frac{1+r_*}{r_*} \left(\frac{\bar c(\alpha-a)}{2-a}\right)^{2-a} (1+|U|)^\frac{\alpha(1-r_*)-2r_*}{\alpha r_*}\right]^\frac\theta{1-\theta},
\eeq 
and, according to \eqref{mt},
\begin{align*}
\theta &=\frac{n r r_*}{nr_*(p-s)+\alpha(r_*(n+p)-n)}=\frac{12(1-a)r_*}{\alpha((5-a)r_*-3)-3ar_*},\\
\mu &=\frac{r+\theta(s-p)}{1-\theta}=\frac{4(1-a)+a\theta}{1-\theta}.
\end{align*}
The positive constant $\bar c$ in \eqref{C1} is the one in \eqref{Sobi} that corresponds to the domain $U\subset \R^3$, number $r_*$ in \eqref{rcond} and $p=2-a$. The fact that $\theta\in(0,1)$ in \eqref{mt2} comes from \eqref{mt}.

We utilize inequality \eqref{imb} in \eqref{diu} with $\varep$ chosen to satisfy 
\beqs
c_{13} \varep \chi_*^{3-2a}=\frac{c_{10}}{4\chi_*}, \text{ i.e., }\varep=\frac{c_{10}}{4c_{13}} \chi_*^{2a-4}.
\eeqs 

It results in
\beq \label{di2}
\begin{aligned}
&\frac{1}{\alpha} \frac{d}{dt}\int_U |\bar{u}|^\alpha dx
\le -\frac{c_{10}(\alpha-1)}{4\chi_*} \int_U (1+u)^{-2} |\nabla \bar u|^{2-a} |\bar{u}|^{\alpha-2}dx\\
& + C_2(\alpha-1) \chi_*^{3-2a} \chi_*^\frac{\theta(4-2a)}{1-\theta}\left\{ \|\bar u\|_{L^\alpha}^{\alpha+\mu+\frac{2\theta}{1-\theta}}+
  \|\bar u\|_{L^\alpha}^{\alpha+\mu} \|1+|\Psi|\|_{L^\frac{2r_*}{1-r_*}}^\frac{2\theta}{1-\theta}\right\} +  c_{15}(\alpha-1) M(t),
\end{aligned} 
\eeq 
where 
$C_2=c_{13}C_1 (4c_{13}/c_{10})^\frac{\theta}{1-\theta}$
and $c_{15}=\max\{1,2c_{11},c_{14}\}$.

Note that the explicit power of $\chi_*$ in the second term on the right-hand side of \eqref{di2} is 
\beqs
3-2a+\frac{\theta(4-2a)}{1-\theta}=\frac{3-2a+\theta}{1-\theta}=\bar\mu.
\eeqs

Multiplying \eqref{di2} by $\alpha$, and using the fact $\alpha(\alpha-1)\ge 2$ for the negative term on the right-hand side, we obtain \eqref{al1} with 
\beq \label{CC}
C_0=\max\{C_2,c_{15}\}\alpha(\alpha-1).
\eeq 
With $a,\theta\in(0,1)$, it is clear that $\mu$, $\bar \mu$, $\gamma$, $\gamma'$ are positive numbers. 
The proof is complete. 
\end{proof}

We are now ready to obtain estimates for $\bar u$ in terms of its initial data and the boundary data $\Psi$, at least for short time.

\begin{theorem}\label{aprio1}
Let $r_*$, $\alpha$, $\theta$, $\gamma$, $\bar\mu$, $C_0$ be as in Lemma \ref{diff1}.  Set 
 \begin{align*} 
 C_*&=C_0\left(1+ 2^\frac{2\theta}{1-\theta}+2^{\alpha+4(1-a)-1}+2^{\frac{\alpha+4(1-a)}{2}-1}\right)(1+|U|)^{\max\left\{1,\frac{(1-r_*)\theta}{r_*(1-\theta)}\right\}},\\ 
 V_0&=1+\int_U |\bar u_0(x) |^{\alpha}dx,\
 V(t)= 1+\int_U |\bar u(x,t)|^{\alpha} dx,\\ 
  \mathcal E(t)&=\chi_*^{\bar\mu}\left(1+\|\Psi(\cdot,t)\|_{L^\frac{2r_*}{1-r_*}}^\frac{2\theta}{1-\theta}\right)
  +\chi_*^{3-2a}\int_U |\Psi(x,t)|^{\alpha+4(1-a)} dx\\
&\quad +\chi_*^{1-a}\int_U |\nabla\Psi(x,t)|^\frac{\alpha+4(1-a)}{2}dx 
+ \chi_*^{-\frac{(3-2a)(\alpha-1)}{5-4a}} \int_U |\Psi_t(x,t)|^\frac{\alpha+4(1-a)}{5-4a}dx.
\end{align*}
 
Suppose there are numbers $T>0$ and $B\in(0,1)$ such that
 \beq\label{Bdef}
 \gamma C_*V_0^\gamma \int_0^T \mathcal E(\tau) d\tau\le B.
\eeq 

Then
\beq\label{uest}
 V(t) \le  V_0\left( 1  - \gamma C_*V_0^\gamma \int_0^t \mathcal E(\tau) d\tau  \right)^{-1/\gamma }
 \text{ for all }t\in[0,T].
\eeq 
Consequently,
\beq\label{u2}
 V(t) \le  V_0(1-B)^{-1/\gamma }
 \text{ for all }t\in[0,T],
\eeq 
\beq\label{mixedterm}
 \int_0^T \int_U (1+u(x,t))^{-2}   |\nabla \bar u(x,t)|^{2-a} |\bar{u}(x,t)|^{\alpha-2}dx dt
 \le \frac{8\chi_* }{c_4}V_0\left(1+\frac{B}{\gamma(1-B)^{1+\frac1\gamma}}\right).
\eeq
\end{theorem}
\begin{proof}
Let $d_0=c_4/(8 \chi_*)$ and 
$$H(t)=\int_U (1+u(x,t))^{-2}   |\nabla \bar u(x,t)|^{2-a} |\bar{u}(x,t)|^{\alpha-2}dx  .
$$

Let $\gamma'$ and $M(t)$ be defined by \eqref{ggam} and \eqref{Mdef} in Lemma \ref{diff1}.
Noticing from \eqref{ggam} that $\gamma'<\gamma$, we estimate
\beqs
\left( \int_U |\bar u(x,t)|^\alpha dx \right)^{1+\gamma} ,
\left(\int_U |\bar u(x,t)|^\alpha dx \right)^{1+\gamma'}\le V(t)^{1+\gamma}.
\eeqs

Combining  this with \eqref{al1}, we have, for $t>0$, that
\beq\label{Vdiff}
\frac {d}{dt} V(t) +d_0H(t)\le C_0\widetilde M(t) V(t)^{1+\gamma},
\eeq
where
$$\widetilde M(t)=\chi_*^{\bar\mu}\left(1+\|1+|\Psi(\cdot,t)|\|_{L^\frac{2r_*}{1-r_*}}^\frac{2\theta}{1-\theta}\right)
+M(t).$$

We find an upper bound for $\widetilde M(t)$ with a simpler expression in $\Psi$.
Using the triangle inequality for the $L^{\frac{2r_*}{1-r_*}}$-norm and \eqref{ee3}, \eqref{ee2},  we estimate
\beqs
\|1+|\Psi(\cdot,t)|\|_{L^\frac{2r_*}{1-r_*}}^\frac{2\theta}{1-\theta}
\le  2^\frac{2\theta}{1-\theta}\left(|U|^\frac{(1-r_*)\theta}{r_*(1-\theta)}+\|\Psi(\cdot,t)\|_{L^\frac{2r_*}{1-r_*}}^\frac{2\theta}{1-\theta}\right),
\eeqs 
\begin{align*}
M(t)&\le \chi_*^{3-2a}\cdot 2^{\alpha+4(1-a)-1}\left( |U|+\int_U |\Psi(x,t)|^{\alpha+4(1-a)} dx\right)\\
&\quad +\chi_*^{1-a}\cdot 2^{\frac{\alpha+4(1-a)}{2}-1}\left (|U|+ \int_U |\nabla\Psi(x,t)|^\frac{\alpha+4(1-a)}{2}dx\right) \\
&\quad + \chi_*^{-\frac{(3-2a)(\alpha-1)}{5-4a}} \int_U |\Psi_t(x,t)|^\frac{\alpha+4(1-a)}{5-4a}dx.
\end{align*}
Regarding the powers of $\chi_*$, observe that $\bar\mu>3-2a>1-a$. Then 
\beq\label{Mtil}
\widetilde M(t)\le \left(1+ 2^\frac{2\theta}{1-\theta}+2^{\alpha+4(1-a)-1}+2^{\frac{\alpha+4(1-a)}{2}-1}\right)(1+|U|)^{\max\left\{1,\frac{(1-r_*)\theta}{r_*(1-\theta)}\right\}}\mathcal E(t).
\eeq 
Therefore, we obtain from \eqref{Vdiff} and \eqref{Mtil} that
\beq\label{Vest1}
\frac {d}{dt} V(t) +d_0 H(t)\le C_*\mathcal E(t) V(t)^{1+\gamma}.
\eeq

Neglecting $d_0 H(t)$ in \eqref{Vest1} and solving the remaining differential  inequality give  \eqref{uest}. 

As a consequence of \eqref{uest}, we have, for $t\in[0,T]$, 
\beqs 
V(t)\le  V_0\left( 1  - \gamma C_*V_0^\gamma \int_0^T \mathcal E(\tau) d\tau  \right)^{-1/\gamma }
\le (1-B)^{-1/\gamma}V_0.
\eeqs
Thus, we obtain \eqref{u2}.

Integrating inequality \eqref{Vest1} in $t$ from $0$ to $T$, we have
\beqs
d_0\int_0^T H(t)dt \le  V_0 + C_* \int_0^T V(t)^{1+\gamma} \mathcal E(t)dt.
\eeqs
Combining this with estimate \eqref{u2} of $V(t)$ and condition  \eqref{Bdef}, we obtain 
\beqs
\begin{split}
d_0\int_0^T H(t)dt &\le V_0+((1-B)^{-1/\gamma}V_0)^{1+\gamma} C_*\int_0^T \mathcal E(t)dt\\
&\le V_0+ (1-B)^{-1-1/\gamma}V_0  \gamma^{-1}B 
= V_0\left(1+\frac{B}{\gamma(1-B)^{1+\frac1\gamma}}\right).
\end{split}
\eeqs
Then estimate \eqref{mixedterm} follows.
\end{proof}

\begin{remark}
Inequality \eqref{mixedterm} gives an indirect estimate for the gradient $\nabla \bar u$, or, in other words, for its weighted $L^{2-a}$-norm with the weight $|\bar u|^{\alpha-2}/(1+u)^2$ depending on the solution $u$.  In \cite{CHK3} when $X=X(y)$ and $\alpha=2$, a similar $L^{2-a}$-estimate (without a weight) was the starting point for other estimates of higher $L^p$-norms of $\nabla u$. They were obtained by the use of Lady\v zenskaja--Ural$'$ceva's iteration \cite{LadyParaBook68}. It is not known whether this method still works for the PDE \eqref{mainuX} with $X=X(z,y)$.
\end{remark}

\section{Estimates for the essential supremum}\label{maxsec}
 
We establish $L^\infty$-estimates for a solution $\bar u$ of \eqref{ubar0} with possibly unbounded initial data.
They will contain some quantities that only involve the boundary data of the following form.

For numbers $q_1,q_2,q_3,q_4>0$ and $T>T'\ge 0$, define
\beq
\begin{aligned}\label{Mzero}
&\mathcal{M}_{T',T}(q_1,q_2,q_3,q_4)=1+ \chi_*^{1-a}\left(\int_{T'}^T \int_U(1+|\nabla\Psi(x,t)|)^{(2-a)q_1}dx dt\right)^\frac{1}{q_1}\\
&\quad +\chi_*^{1-a}\left(\int_{T'}^T \int_U(1+|\Psi(x,t)|)^{2(1-a)q_2}(1+|\nabla\Psi(x,t)|)^{(2-a)q_2} dx dt\right)^\frac{1}{q_2}\\
&\quad +\chi_*^{3-a}\left(\int_{T'}^T \int_U(1+|\Psi(x,t)|)^{\tilde q q_3} dxdt\right)^\frac{1}{q_3}
+\left(\int_{T'}^T \int_U|\Psi_t|^{q_4} dxdt\right)^\frac{1}{q_4},
\end{aligned}
\eeq 
where $\tilde q=2\max\{2-a,3-2a\}$.

We use Moser's iteration and have technical preparations with key inequalities in Lemmas \ref{GLk} and \ref{iter1} below.
In the following, $Q_T$ denotes the cylinder $U\times(0,T)$ in $\R^4$, and $|Q_T|$ denotes its $4$-dimensional Lebesgue measure.

\begin{lemma} \label{GLk}
Assume numbers $\tilde\kappa$ and $p_i$, for $i=1,2,3,4$, satisfy
\beq\label{kp} 
\tilde \kappa>p_1,p_2,p_3,p_4>1.
\eeq 
For $i=1,2,3,4$, let $q_i$ be the H\"older conjugate exponent of $p_i$, that is, 
\beq\label{piqi} 1/p_i+1/q_i=1\text{ for }i=1,2,3,4.
\eeq 

Let $T>T_2>T_1\ge 0$. If
\beq\label{alp-Large}
\alpha \ge  \max\left \{2,\frac{2 p_3 (1-a)}{\tilde\kappa -p_3}, \frac{4(1-a)}{\tilde\kappa -1}\right \},
\eeq 
then one has
\beq\label{Sest2}
 \sup_{t\in[T_2,T]} \int |\bar u|^\alpha(x,t) dx 
 \le \alpha^2 K  \mathcal{M}_0 (\| \bar u\|_{L^{\tilde\kappa \alpha}(U\times(T_1,T))}^{\alpha-2}+ \| \bar u\|_{L^{\tilde\kappa \alpha}(U\times(T_1,T))}^{\alpha+4(1-a)}),
 \eeq 
 \beq  \label{Sest3}
  \int_{T_2}^T\int_U (1+u)^{-2}|\nabla \bar u|^{2-a}|\bar u|^{\alpha-2} dx dt
\le \frac{2\chi_*}{c_{10}}K \mathcal{M}_0(\| \bar u\|_{L^{\tilde\kappa \alpha}(U\times(T_1,T))}^{\alpha-2}+ \| \bar u\|_{L^{\tilde\kappa \alpha}(U\times(T_1,T))}^{\alpha+4(1-a)}),
\eeq
where $\mathcal{M}_0=\mathcal{M}_{T_1,T}(q_1,q_2,q_3,q_4)$ and
\beq\label{Kdef}
K=c_{16} (1+|Q_T|)\Big(1+\frac1{T_2-T_1}\Big) \text{ with }c_{16} = 9(1+c_{11}+c_{12}).
\eeq 
\end{lemma}
\begin{proof}
 Let $\zeta=\zeta(t)$ be a $C^1$-function  on $[0,T]$ with 
\begin{align} \notag
&\zeta(t)=0 \text{ for } 0\le t\le T_1,\  
\zeta(t)=1 \text{ for } T_2\le t\le T,\\
\label{xiprop} 
& 0\le \zeta(t)\le 1\text{ and } 0\le \zeta'(t)\le \frac{2}{T_2-T_1} \text{ for }0\le t\le T.
\end{align} 

Multiplying the PDE in \eqref{ubar0}  by $|\bar u|^{\alpha-2}\bar u \zeta^2(t)$, integrating over $U$, and integrating by parts  give
\begin{align*}
&\frac{1}{\alpha}\frac{d}{dt}\int_U |\bar u|^\alpha\zeta^2 dx -\frac{1}{\alpha}\int_U 2|\bar u|^{\alpha} \zeta \zeta' dx\\
 &\quad = -(\alpha-1)\int_U X(u, \Phi(x,t)) | \bar u |^{\alpha-2} \nabla \bar u  \zeta^2 dx
     -   \int_U\Psi_t|\bar u|^{\alpha-2}\bar u \zeta^2 dx.
\end{align*}

Noticing that the function $\zeta=\zeta(t)\ge 0$ is independent of $x$, we have, same as inequality \eqref{di1}, 
\beqs
\frac{1}{\alpha}\frac{d}{dt}\int_U |\bar u|^\alpha\zeta^2 dx -\frac{2}{\alpha}\int_U |\bar u|^{\alpha} \zeta \zeta' dx
  \le (\alpha-1)( -\tilde I_0+\tilde I_1+\tilde I_2+\tilde I_3+\tilde I_4)+\tilde I_5,
\eeqs 
where $\tilde I_i = I_i\zeta^2$ for $i=0,1,\ldots,5$. 
Then, similar to \eqref{ddtpf1},
\beq\label{ddtpf2}
\begin{aligned}
\frac{1}{\alpha}\frac{d}{dt}\int_U |\bar{u}|^\alpha \zeta^2 dx-\frac{2}{\alpha}\int_U |\bar u|^{\alpha} \zeta \zeta' dx
&\le -\frac{c_{10}}2\chi_*^{-1}(\alpha-1)\int_U (1+u)^{-2} |\nabla \bar u|^{2-a} |\bar{u}|^{\alpha-2}\zeta^2 dx \\
&\quad + (\alpha-1)(\tilde J_0+\tilde I_1+\tilde J_2+\tilde J_3+\tilde J_4+\tilde I_6)+\tilde I_5,
\end{aligned}
\eeq 
where  $\tilde J_i = J_i\zeta^2$ for $i=0,2,3,4$, with $\varep$ and $\delta$ particularly chosen in \eqref{ed}, and $\tilde I_6 = I_6\zeta^2 $.  

On the one hand, neglecting the negative term on the right-hand side of \eqref{ddtpf2} and integrating the resulting inequality  in time from $0$ to $t$, for $t\in[T_2,T]$, with the use of the fact $\zeta(0)=0$, and then taking the supremum in $t$ over $[T_2,T]$,  we obtain
\beq\label{supone}
\begin{aligned}
\frac{1}{\alpha}\sup_{t\in[T_2,T]}\int_U |\bar u(x,t)|^\alpha dx
=\frac{1}{\alpha}\sup_{t\in[T_2,T]}\int_U |\bar u(x,t)|^\alpha\zeta^2(t) dx
\le \mathcal J,
\end{aligned}
\eeq 
where
\beq\label{EE}
\mathcal J=(\alpha-1)\int_0^T(\tilde J_0+\tilde I_1+\tilde J_2+\tilde J_3+\tilde J_4+\tilde I_6)dt+\int_0^T\tilde I_5 dt
+ \frac{2}{\alpha}\iint_{Q_T} |\bar u|^{\alpha} \zeta \zeta' dxdt.
\eeq

On the other hand, integrating \eqref{ddtpf2} in $t$ from $0$ to $T$ gives
\beqs
\frac{c_{10}(\alpha-1)}{2 \chi_*}\iint_{Q_T}(1+u)^{-2} |\nabla \bar u|^{2-a}|\bar{u}|^{\alpha-2} \zeta^2dxdt\\
\le \mathcal J.
\eeqs
Hence,
\beq \label{suptwo}
\begin{aligned}
&\int_{T_2}^T\int_U (1+u)^{-2} |\nabla \bar u|^{2-a}|\bar{u}|^{\alpha-2} dxdt
=\int_{T_2}^T\int_U (1+u)^{-2} |\nabla \bar u|^{2-a}|\bar{u}|^{\alpha-2} \zeta^2 dxdt\\
&\le \iint_{Q_T}(1+u)^{-2} |\nabla \bar u|^{2-a}|\bar{u}|^{\alpha-2} \zeta^2dxdt
\le \frac{2\chi_*}{c_{10}(\alpha-1)}\mathcal J.
\end{aligned} 
\eeq 

We focus on estimating the quantity $\mathcal J$ now. Define $Y(\alpha)=\iint_{Q_T} |\bar u|^{\alpha} \zeta dxdt$.

Using the fact $0\le \zeta^2\le\zeta$ and  previous estimate \eqref{JIs}, we have
\begin{align*}
\int_0^T (\tilde J_0+\tilde I_1+\tilde J_4+\tilde I_6)dt 
&\le \int_0^T (J_0+I_1+J_4+I_6)\zeta dt \\
&\le c_{11}\chi_*^{1-a}\iint_{Q_T}  |\bar{u}|^{\alpha-2a}(1+|\nabla\Psi|)^{2-a}  \zeta dx dt\\
&\quad +c_{11}\chi_*^{1-a}\iint_{Q_T} |\bar{u}|^{\alpha-2}(1+|\Psi|)^{2(1-a)}(1+|\nabla\Psi|)^{2-a} \zeta dx dt.
\end{align*} 

On the right-hand side of the preceding inequality, by applying H\"older's inequality with powers $p_1,q_1$ to the first integral on the right-hand side, and with powers $p_2,q_2$ to the second integral, we obtain
\beq\label{tJI}
\int_0^T (\tilde J_0+\tilde I_1+\tilde J_4+\tilde I_6)dt 
\le c_{11}\chi_*^{1-a}[ Y(p_1(\alpha-2a))^\frac1{p_1}E_1+Y(p_2(\alpha-2))^\frac1{p_2}E_2],
\eeq 
where
\begin{align*}
E_1&= \left(\iint_{Q_T}(1+|\nabla\Psi|)^{(2-a)q_1}\zeta dx dt\right)^\frac{1}{q_1},\\  E_2&=\left(\iint_{Q_T}(1+|\Psi|)^{2(1-a)q_2}(1+|\nabla\Psi|)^{(2-a)q_2}  \zeta dx dt\right)^\frac{1}{q_2}.
\end{align*}

Similarly, by the fact $0\le \zeta^2\le\zeta$ and estimates \eqref{pJ2}, \eqref{pJ3},  we have
\beq\label{pJ}
\begin{aligned}
&\int_0^T (\tilde J_2+\tilde J_3)dt
\le \int_0^T J_2\zeta dt + \int_0^T J_3 \zeta dt\\
&\le c_{12}\chi_*^{3-2a}\iint_{Q_T} \left\{  |\bar{u}|^{\alpha+4(1-a)} 
+  |\bar{u}|^{\alpha}|\Psi|^{2(1-a)}
+  |\bar{u}|^{\alpha+2-2a}(1+|\Psi|)^{2(1-a)}\right\}\zeta dxdt\\
&\quad + c_{12}\chi_*^{3-2a}\iint_{Q_T} \left\{  |\bar{u}|^{\alpha-2a}(1+|\Psi|)^{2(2-a)}
+  |\bar{u}|^{\alpha-2}(1+|\Psi|)^{2(3-2a)}\right\}\zeta dxdt.
\end{aligned}
\eeq 

Note that each power of $|\Psi|$ or $(1+|\Psi|)$ in \eqref{pJ} are less than or equal to $\tilde q$.
Then
\begin{align*}
&\int_0^T (\tilde J_2+\tilde J_3)dt
\le c_{12}\chi_*^{3-2a}\Big\{ \iint_{Q_T}  |\bar{u}|^{\alpha+4(1-a)} \zeta dxdt
+\iint_{Q_T}   |\bar{u}|^{\alpha}(1+|\Psi|)^{\tilde q}\zeta dxdt\\
&\quad + \iint_{Q_T}  |\bar{u}|^{\alpha+2-2a}(1+|\Psi|)^{\tilde q}\zeta dxdt
+\iint_{Q_T} |\bar{u}|^{\alpha-2a}(1+|\Psi|)^{\tilde q}\zeta dxdt\\
&\quad + \iint_{Q_T}  |\bar{u}|^{\alpha-2}(1+|\Psi|)^{\tilde q}\zeta dxdt\Big\}.
\end{align*}

Applying H\"older's inequality with powers $p_3$ and $q_3$ to the last four integrals yields
\beq\label{pJJ}
\begin{aligned}
\int_0^T (\tilde J_2+\tilde J_3)dt
&\le c_{12}\chi_*^{3-2a}\Big\{ Y(\alpha+4(1-a)) +Y(p_3\alpha)^\frac{1}{p_3}E_3
+Y(p_3(\alpha+2-2a))^\frac{1}{p_3}E_3\\
&\quad +Y(p_3(\alpha-2a))^{\frac{1}{p_3}}E_3+Y(p_3(\alpha-2))^\frac{1}{p_3}E_3\Big\},
\end{aligned}
\eeq 
where
\beqs
E_3=\left(\iint_{Q_T}(1+|\Psi|)^{\tilde q q_3}\zeta dxdt\right)^\frac{1}{q_3}.
\eeqs

Next, by the fact $0\le \zeta^2\le\zeta$ and H\"older's inequality with powers $p_4$ and $q_4$,
\beq\label{tI5}
    \int_0^T\tilde I_5dt\le  \iint_{Q_T}|\bar{u}|^{\alpha-1}|\Psi_t|\zeta dxdt\le  Y(p_4(\alpha-1))^\frac{1}{p_4}E_4,
\eeq 
where
\beqs
    E_4=\left(\iint_{Q_T}|\Psi_t|^{q_4}\zeta dxdt\right)^\frac{1}{q_4}.
\eeqs 

Finally, for the last term of $\mathcal J$ in \eqref{EE}, using the second property of \eqref{xiprop}, we have 
\beq \label{nzz}
    \frac{2}{\alpha}\iint_{Q_T} |\bar u|^{\alpha} \zeta \zeta' dxdt\le  \frac{4}{\alpha(T_2-T_1)} Y(\alpha)
    \le  \frac{2}{T_2-T_1} Y(\alpha).
\eeq

Combining formula \eqref{EE} with the  above estimates \eqref{tJI}, \eqref{pJJ}, \eqref{tI5}, \eqref{nzz}  yields 
\beq\label{Imain5}
\begin{aligned}
&\mathcal J
\le (\alpha-1)c_{11}\chi_*^{1-a}\left[  Y(p_1(\alpha-2a))^\frac1{p_1}E_1+Y(p_2(\alpha-2))^\frac1{p_2}E_2\right]\\
&\quad +(\alpha-1)c_{12}\chi_*^{3-2a}\Big\{ Y(\alpha+4(1-a)) +\Big[Y(p_3\alpha)^\frac{1}{p_3}
+Y(p_3(\alpha+2-2a))^\frac{1}{p_3}\\
&\quad +Y(p_3(\alpha-2a))^{\frac{1}{p_3}}+Y(p_3(\alpha-2))^\frac{1}{p_3}\Big]E_3\Big\} + Y(p_4(\alpha-1))^\frac{1}{p_4}E_4
 +\frac{2}{T_2-T_1} Y(\alpha).
\end{aligned}
\eeq

Define $Y_*=\|\bar{u}\|_{L^{\tilde{\kappa}\alpha}(U\times(T_1,T))}$. If $0<\beta \le \tilde \kappa \alpha$ then, by H\"older's inequality,
\beq\label{X}
Y(\beta)=\iint_{Q_T} |\bar u|^\beta \zeta dxdt \le \int_{T_1}^T \int_U |\bar u|^\beta  dxdt \le  Y_*^\beta |Q_T|^{1-\frac{\beta}{\tilde \kappa \alpha}} \le Y_*^\beta (1+|Q_T|).
\eeq

Under conditions \eqref{kp} and \eqref{alp-Large}, one has  
\beqs 
p_1(\alpha-2a), p_2(\alpha-2), p_4(\alpha-1) < \tilde \kappa\alpha \text{ and }
\alpha+4(1-a), p_3(\alpha+2-2a)\le \tilde \kappa\alpha .
\eeqs 

Thus, we have from \eqref{Imain5}, \eqref{X} that
\begin{align*}
\mathcal J&\le (\alpha-1)(1+|Q_T|)c_{11}\chi_*^{1-a}\left [Y_*^{\alpha-2a}E_1+Y_*^{\alpha-2}E_2\right ]\\
&\quad +(\alpha-1)(1+|Q_T|)c_{12}\chi_*^{3-2a}\{Y_*^{\alpha+4(1-a)}+ [Y_*^{\alpha}+Y_*^{\alpha+2-2a}+Y_*^{\alpha-2a}+Y_*^{\alpha-2}]E_3\}\\
&\quad + (1+|Q_T|) Y_*^{\alpha-1}E_4
 +\frac{2}{T_2-T_1} (1+|Q_T|) Y_*^{\alpha}.
\end{align*}
It follows that
\beqs
\mathcal J\le (1+c_{11}+c_{12})(1+|Q_T|)(\alpha-1)\Big(1+\frac {1}{T_2-T_1}\Big)M_0  J,
\eeqs
 where
 \begin{align*}
M_0&= 1+\chi_*^{1-a}(E_1+E_2)+\chi_*^{3-a}E_3+E_4,\\
J &= 2Y_*^{\alpha-2}+2Y_*^{\alpha-2a}+Y_*^{\alpha-1}+3Y_*^{\alpha}+Y_*^{\alpha +2(1-a)}+Y_*^{\alpha +4(1-a)}.
\end{align*} 

Because $0\le\zeta\le 1$ on $[0,T]$ and $\zeta=0$ on $[0,T_1]$, we have $M_0\le \mathcal M_0$.

Thanks to inequality \eqref{ee4}, one has
$$Y_*^{\alpha-2a},Y_*^{\alpha-1},Y_*^{\alpha},Y_*^{\alpha +2(1-a)}
\le Y_*^{\alpha-2} +Y_*^{\alpha+4(1-a)},$$
hence, $J \le 9(Y_*^{\alpha-2} +Y_*^{\alpha+4(1-a)})$. 
Therefore, 
\beq\label{I0i}
\mathcal J \le (\alpha-1)K \mathcal{M}_0(Y_*^{\alpha-2} +Y_*^{\alpha+4(1-a)}). 
\eeq

Then estimate \eqref{Sest2} follows \eqref{supone} and \eqref{I0i}, while estimate \eqref{Sest3} follows \eqref{suptwo} and \eqref{I0i}. The proof is complete.
\end{proof}

\begin{lemma}\label{iter1} 
Let $r_*$ satisfy \eqref{rcond} and set $\lambda_0= (r_*(5-a)-3)/(3r_*)$.
Assume \eqref{kp}, \eqref{piqi},  
\beq\label{alph}
\alpha \ge  \max\left \{2,\frac{2p_3(1-a)}{\tilde\kappa -p_3}, \frac{4(1-a)}{\tilde\kappa -1}\right \} \text{ and } \alpha>\frac{a}{\lambda_0}.
\eeq 

Let 
\beq\label{newkappa}
\kappa=1+\lambda_0-\frac{a}{\alpha}, \quad 
\tilde \theta = \frac{1}{1+ \frac{\lambda_0 \alpha}{\alpha-a} },\quad
\mu_1=1+\frac{\tilde \theta a}{\alpha-a}.
\eeq 

If  $T >T_2 >T_1\ge 0$ then
 \beq\label{g2}
 \|\bar u\|_{L^{\kappa \alpha}(U\times(T_2,T))}\le (A_\alpha B_\alpha)^\frac1\alpha
 \left( \|\bar u\|_{L^{\tilde\kappa \alpha}(U\times(T_1,T))}^{\nu_1}+\|\bar u\|_{L^{\tilde\kappa \alpha}(U\times(T_1,T))}^{\nu_2} \right)^\frac1\alpha,
 \eeq
 where 
$\nu_1=(\alpha-2)\mu_1$, $\nu_2=(\alpha+4(1-a))\mu_1$,
 \begin{align}\label{Aalp}
 A_\alpha&=  2^{\mu_1-1+\frac{1}{\kappa}(2+\frac{1}{r_*})}
 \left[\frac{1}{c_{10}}\left( \frac{\bar c}{2-a}\right)^{2-a}\right ]^\frac{1}{\kappa}
 c_{16}^{\mu_1},\\
 \label{Balp}
B_\alpha&=
 \chi_*^\frac{1}{\kappa}
 \alpha^{2\mu_1-\frac{a}{\kappa}} \hat E^{\frac{1}{\kappa}} \left[(1+ |Q_T|)\Big(1+\frac1{T_2-T_1}\Big) \mathcal{M}_0\right]^{\mu_1}
 \end{align}
 with $\mathcal M_0=\mathcal{M}_{T_1,T}(q_1,q_2,q_3,q_4)$ and 
 \beqs 
 \hat E=\esssup_{t\in(T_2,T)}\|1+|\Psi(\cdot,t)|\|_{L^\frac{2 r_*}{1-r_*}}^2  + \esssup_{t\in(T_2,T)}\|\bar{u}(\cdot,t)\|_{L^\frac{2 r_*}{1-r_*}}^2.
\eeqs 
\end{lemma}
\begin{proof} 
We apply Lemma \ref{PS3}\ref{ii26} to  $n=3$, $p=2-a$, $s=2$, $\beta=2$, and functions $u:=\bar u$, $\varphi:=\Psi$, $v:=\bar u+\Psi=u$,
and the interval $(T_2,T)$ in place of $(0,T)$.
Note, from \eqref{mdef}, that  $m=(\alpha-a)/(2-a)$.
Same as in Step 2 in the proof of Lemma \ref{diff1}, condition \eqref{powcond} becomes \eqref{rcond}.
Clearly, condition \eqref{alcond} becomes 
\beqs
\alpha \ge 2 \text{ and } \alpha>\frac{a}{\lambda_0},
\eeqs 
which is satisfied thanks to \eqref{alph}.
Then, by inequality  \eqref{ppsi3}, one has
\beq\label{beginest2}
\|\bar{u}\|_{L^{\kappa \alpha}(U\times(T_2,T))}
\le \hat{C}^{\frac1{\kappa\alpha}} \hat E^\frac1{\kappa\alpha}
\esssup_{t\in(T_2,T)}\|\bar{u}(\cdot,t)\|_{L^\alpha}^{1-\tilde\theta}\left(\int_{T_2}^T\int_U |\bar{u}|^{\alpha-2}|\nabla \bar{u}|^{2-a} (1+|u|)^{-2}dxdt\right)^\frac1{\kappa\alpha},
\eeq
where $\hat C=2^{2+\frac{1-r_*}{r_*}}(\bar c\cdot \frac{\alpha-a}{2-a})^{2-a}$,  the numbers $\kappa$  and $\tilde \theta$ are given in \eqref{defkappa}, which assume the values in \eqref{newkappa} now. 

We estimate the right-hand side of \eqref{beginest2} by Lemma \ref{GLk}. Note that condition \eqref{alp-Large} is the first part of \eqref{alph}.
Recalling that $K$ is defined in \eqref{Kdef}, we denote 
\beqs
Y_*= \|\bar u\|_{L^{\tilde \kappa\alpha}(U\times(T_1,T))},\quad 
M_1=\alpha^2K \mathcal{M}_0, \quad 
S=M_1 (Y_*^{\alpha-2} + Y_*^{\alpha+4(1-a)}).
\eeqs  

By estimates \eqref{Sest2} and \eqref{Sest3} in Lemma \ref{GLk}, we have 
\beq\label{Szeroest2}
\esssup_{t\in(T_2,T)}\int_U |\bar u|^\alpha dx\le S,\quad 
\int_{T_2}^T\int_U |\bar u|^{\alpha-2}|\nabla \bar u|^{2-a} (1+|u|)^{-2}dxdt\le \frac{2\chi_*}{c_{10}\alpha^2}S.
\eeq 

Then combining \eqref{beginest2} and \eqref{Szeroest2} yields 
\begin{align}
\| \bar u\|_{L^{\kappa\alpha}(U\times (T_2,T))} 
&\le \left (\frac{2\chi_*}{c_{10}\alpha^2}\hat C \hat E\right )^\frac{1}{\kappa\alpha} 
S^{\frac1\alpha(1-\tilde\theta + \frac1{\kappa})} \notag \\
&=\left\{  \left(\frac{2\chi_*}{c_{10}\alpha^2}\hat C \hat E\right )^\frac{1}{\kappa} 
M_1^{1-\tilde\theta + \frac1{\kappa}} 
\left(Y_*^{\alpha-2} + Y_*^{\alpha+4(1-a)}\right)^{1-\tilde\theta + \frac1{\kappa}} \right\}^{\frac1\alpha}.\label{uY}
\end{align}

Using the formula of $\kappa$ in \eqref{mptil}, we have
\beqs 
    \frac1\kappa-\tilde \theta=\frac{\tilde\theta\alpha}{mp}-\tilde \theta=\frac{\tilde\theta\alpha}{\alpha-s+p}-\tilde \theta=\frac{\tilde \theta(s-p)}{\alpha-s+p}=\frac{\tilde \theta a}{\alpha-a}.
\eeqs
It follows that the power $1-\tilde\theta + \frac1{\kappa}$ in \eqref{uY} is exactly the number $\mu_1>1$ defined in \eqref{newkappa}.
Applying inequality \eqref{ee3} to $(Y_*^{\alpha-2} + Y_*^{\alpha+4(1-a)})^{\mu_1}$ in \eqref{uY}, we obtain
\beq\label{nJ}
\begin{aligned}
\| \bar u\|_{L^{\kappa\alpha}(U\times (T_2,T))} 
&\le \left\{ \left (\frac{2\chi_*}{c_{10}\alpha^2}\hat C \hat E\right )^\frac{1}{\kappa} M_1^{\mu_1} 
2^{\mu_1-1} (Y_*^{(\alpha-2)\mu_1}  + Y_*^{(\alpha+4(1-a))\mu_1})\right\}^{\frac 1 \alpha}\\
&= M_2^\frac 1{\alpha}(Y_*^{\nu_1}  + Y_*^{\nu_2})^{\frac 1 \alpha},
\text{ where }M_2 = 2^{\mu_1-1}\left(\frac{2\chi_*}{c_{10}\alpha^2}\hat C \hat E\right)^\frac{1}{\kappa} M_1^{\mu_1}.
\end{aligned}
\eeq

Simple bound $\alpha-a<\alpha$ in the formula of $\hat C$, and elementary calculations give
\beq\label{MAal}
\begin{aligned}
M_2 &\le 2^{\mu_1-1}\left [\frac{2^{3+\frac{1-r_*}{r_*}}\chi_*}{c_{10}\alpha^2}\left(\frac{\bar c\alpha}{2-a}\right)^{2-a} \hat E\right]^\frac{1}{\kappa}\\
&\quad\times\left [c_{16}(1+|Q_T|)\Big(1+\frac1{T_2-T_1}\Big)\alpha^2 \mathcal{M}_0\right ]^{\mu_1}=A_\alpha B_\alpha.
\end{aligned}
\eeq
Therefore, we obtain \eqref{g2} from \eqref{nJ} and \eqref{MAal}.
\end{proof}

We simplify inequality \eqref{g2} to make it more suitable to the Moser iteration below.

Firstly, observe that $1<\mu_1<1+a/(\alpha-a)$, and, hence, the powers $\nu_1$ and $\nu_2$ in \eqref{g2} can be simply bounded by
$\nu_3<\nu_1<\nu_2<\nu_4$, where 
\beq \label{nns}
 \nu_3=\nu_{3,\alpha}\eqdef \alpha-2\text{ and }\nu_4=\nu_{4,\alpha}\eqdef(\alpha+4(1-a))\left( 1 + \frac{a}{\alpha-a}\right).
\eeq 
Then, thanks to \eqref{ee4},
\beq\label{g5}
\|\bar u\|_{L^{\tilde\kappa \alpha}(U\times(T_1,T))}^{\nu_1}+\|\bar u\|_{L^{\tilde\kappa \alpha}(U\times(T_1,T))}^{\nu_2}
\le 2\left(\|\bar u\|_{L^{\tilde\kappa \alpha}(U\times(T_1,T))}^{\nu_3}+\|\bar u\|_{L^{\tilde\kappa \alpha}(U\times(T_1,T))}^{\nu_4}\right).
\eeq 

Secondly, using the facts $\mu_1\le 2$ and $\kappa\ge 1$, we estimate $A_\alpha$ in \eqref{Aalp} by
\beq\label{Aa}
 A_\alpha\le 2^{3+\frac{1}{r_*}}\bar c_1c_{16}^2, \text{ where } 
 \bar c_1=\max\left\{1,\frac{1}{c_{10}}\left( \frac{\bar c}{2-a}\right)^{2-a}\right\}.
\eeq 

Thirdly, we estimate $B_\alpha$ given by formula \eqref{Balp}. Regarding the powers in that formula,  note  that 
$\tilde \theta \le \frac{1}{1+\lambda_0},$
then 
\beq\label{mm2}
\mu_1\le 1+\tilde\theta a\le \mu_2\eqdef 1+\frac{a}{1+\lambda_0}.
\eeq  
Property \eqref{mm2} and the fact $\kappa\le 1+\lambda_0$ yield that the power of $\alpha$ satisfies 
\beq\label{mm3}
2\mu_1-\frac{a}{\kappa}\le  2\mu_2-\frac{a}{1+\lambda_0}=\mu_3\eqdef 2+\frac{a}{1+\lambda_0}.
\eeq
Concerning the remaining power $1/\kappa$, one has, thanks to the fact $\alpha\ge 2$, that
\beq \label{khat}
\kappa\ge \hat\kappa\eqdef 1+\lambda_0-\frac{a}{2}.
\eeq 
Therefore,
\beq \label{Ba}
B_\alpha\le 
 \chi_*^{1/\hat\kappa}
 \alpha^{\mu_3} \bar E^{1/\hat\kappa} \left[(1+ |Q_T|)\Big(1+\frac1{T_2-T_1}\Big) \mathcal{M}_0\right]^{\mu_2},
\text{ where }\bar E =\max\{1,\hat E\}.
\eeq 

Fourthly, assume 
\beq \label{kkc}
\kappa\ge \kappa'>1.
\eeq 
By H\"older's inequality,
\beq\label{g4}
\begin{aligned}
\|\bar u\|_{L^{\kappa' \alpha}(U\times(T_2,T))}
&\le |Q_T|^{(\frac1{\kappa'}-\frac1{\kappa})\frac1{\alpha}}\|\bar u\|_{L^{\kappa\alpha}(U\times(T_2,T))}
\le (1+|Q_T|)^\frac1{\alpha}\|\bar u\|_{L^{\kappa\alpha}(U\times(T_2,T))}.
\end{aligned}
\eeq

Combining \eqref{g4} with \eqref{g2}, and making use of estimates \eqref{g5}, \eqref{Aa}, \eqref{Ba} yield 
 \beq\label{g3}
 \|\bar u\|_{L^{\kappa' \alpha}(U\times(T_2,T))}\le \left\{ \bar A \bar B \alpha^{\mu_3}
 \left( \|\bar u\|_{L^{\tilde\kappa \alpha}(U\times(T_1,T))}^{\nu_3}+\|\bar u\|_{L^{\tilde\kappa \alpha}(U\times(T_1,T))}^{\nu_4} \right)\right\}^\frac1\alpha,
 \eeq
 where
\beq\label{Abound}
 \bar A=  2^{4+\frac{1}{r_*}}\bar c_1  c_{16}^2,
 \quad \bar B=\chi_*^{1/\hat\kappa} \bar E^{1/\hat\kappa}(1+ |Q_T|)^{1+\mu_2}\Big(1+\frac1{T_2-T_1}\Big)^{\mu_2} \mathcal{M}_0^{\mu_2}.
\eeq

Obviously, \eqref{g3} is only useful when $\kappa'>\tilde\kappa$, which will be satisfied in Theorem \ref{maxest} below.

Next, we recall a lemma on numeric sequences that will be used in our version of  Moser's iteration.

\begin{lemma}[\cite{CHK1}, Lemma A.2]\label{Genn}
Let $y_j\ge 0$, $\kappa_j>0$, $s_j \ge r_j>0$ and $\omega_j\ge 1$  for all $j\ge 0$.
Suppose there is $A\ge 1$ such that
\beq\label{yineq}
y_{j+1}\le A^\frac{\omega_j}{\kappa_j} (y_j^{r_j}+y_j^{s_j})^{\frac 1{\kappa_j}}\quad\forall j\ge 0.
\eeq
Denote $\beta_j=r_j/\kappa_j$ and $\gamma_j=s_j/\kappa_j$.
Assume
\beqs
\bar\alpha \eqdef \sum_{j=0}^{\infty}  \frac{\omega_j}{\kappa_j}<\infty\text{ and the products }
\prod_{j=0}^\infty \beta_j, 
\prod_{j=0}^\infty \gamma_j\text{ converge to positive numbers } \bar\beta,\bar \gamma, \text{ resp.}
\eeqs
Then
\beqs
y_j\le (2A)^{G_{j} \bar \alpha} \max\Big\{  y_0^{\gamma_0\ldots\gamma_{j-1}},y_0^{\beta_0\ldots\beta_{j-1}} \Big\}\quad \forall j\ge 1,
\eeqs
where $G_j=\max\{1, \gamma_{m}\gamma_{m+1}\ldots\gamma_{n}:1\le m\le n< j\}$.
Consequently,
\beq\label{doub2}
\limsup_{j\to\infty} y_j\le (2A)^{G \bar\alpha} \max\{y_0^{\bar \gamma},y_0^{\bar \beta}\},
\quad\text{where  }G=\limsup_{j\to\infty} G_j.
\eeq
\end{lemma}

Some conditions on involved parameters will be imposed and are summarized here.

\begin{assumption}\label{asmp44}
Let number $r_*$ satisfy \eqref{rcond} and set $\lambda_0= (r_*(5-a)-3)/(3r_*)$.
Fix a number $\tilde\kappa\in(1,\sqrt{1+\lambda_0})$ and let $p_i>1$, $q_i>1$, for $i=1,2,3,4$, satisfy \eqref{kp} and \eqref{piqi}.
\end{assumption}

We obtain the first estimate for the essential supremum of $\bar u(x,t)$.

\begin{theorem}\label{maxest}
Under Assumption \ref{asmp44}, let $\alpha_0$ be a positive number such that 
 \beq\label{xal}
 \alpha_0 \ge \max\left\{ \frac{2p_3(1-a)}{\tilde\kappa-p_3},\frac{4(1-a)}{\tilde\kappa-1},\frac{a}{1+\lambda_0-\tilde\kappa^2} \right\}\text{ and } \alpha_0>\max\left\{2,\frac{a}{\lambda_0}\right\}.
 \eeq 

There are positive constants $\tilde\mu<\tilde\nu$ and $\omega$, which can be identified by \eqref{mnutil} and  \eqref{oG} below,  such that if $T>0$ and $\sigma\in(0,1)$ then
\beq\label{maxineq}
\begin{aligned}
 \|\bar u\|_{L^\infty(U\times(\sigma T,T))}
 &\le \left[ \bar c_2 \alpha_0^{\mu_3}\chi_*^{1/\hat\kappa} \Big(1+\frac1{\sigma T}\Big)^{\mu_2}(1+|Q_T|)^{\mu_3} \mathcal M_1^{\mu_2}
E_*^{1/\hat\kappa}\right]^\omega \\
& \quad\times \max\left\{  \|\bar u\|_{L^{\tilde \kappa\alpha_0}(U\times(0,T))}^{\tilde\mu} , \|\bar u\|_{L^{\tilde \kappa\alpha_0}(U\times(0,T))}^{\tilde\nu}\right\},
\end{aligned}
\eeq
 where  $\bar c_2=(2\tilde\kappa)^{\mu_3}\bar{A}$, numbers 
$\mu_2$, $\mu_3$, $\hat\kappa$, $\bar A$ are respectively given in \eqref{mm2}, \eqref{mm3}, \eqref{khat}, \eqref{Abound},
 $\mathcal M_1=\mathcal{M}_{0,T}(q_1,q_2,q_3,q_4)$ and 
 \beq\label{Estar}
E_*=\max\left\{ 1, \esssup_{t\in(\sigma T/2,T)}\|1+|\Psi(\cdot,t)|\|_{L^\frac{2 r_*}{1-r_*}}^2  + \esssup_{t\in(\sigma T/2,T)}\|\bar{u}(\cdot,t)\|_{L^\frac{2 r_*}{1-r_*}}^2\right\}.
 \eeq
\end{theorem}
\begin{proof}
We prove \eqref{maxineq} by adapting Moser's iteration. We iterate inequality \eqref{g3} with suitable parameters. 
For our convenience, re-denote $\kappa$ defined in \eqref{newkappa} by $\kappa_\alpha$. Then $\kappa_\alpha$ is  increasing in $\alpha$.

Set $\beta_j=\tilde \kappa^j\alpha_0$  for $j\ge 0$.
Since $\tilde \kappa>1$, the sequence $(\beta_j)_{j=0}^\infty$ is  increasing. In particular, 
\beq\label{bej2}
\beta_j\ge \beta_0=\alpha_0\text{ for all }j\ge 0.
\eeq 
This relation and \eqref{xal} imply that $\alpha=\beta_j$ satisfies condition \eqref{alph}, and 
\beq\label{kak}
\kappa_{\beta_j}\ge \kappa_{\alpha_0}= 1+\lambda_0  -\frac{a}{\alpha_0}\ge \tilde\kappa^2.
\eeq

Set $\kappa'=\tilde\kappa^2>1$. Then property \eqref{kak} implies that \eqref{kkc} holds for $\alpha=\beta_j$.

For $j\ge 0$, let $t_j=\sigma T(1-\frac 1{2^{j}})$. Then $t_0=0$, $t_1=\sigma T/2$, $t_j$ is strictly increasing, and $t_j\to \sigma T$ as $t\to\infty$.
 
For $j\ge 0$, applying inequality \eqref{g3} to $\alpha=\beta_{j}$, $T_1=t_{j}$ and $T_2=t_{j+1}$, we have 
\begin{align}
\| \bar u\|_{L^{\kappa'\beta_{j}}(U\times (t_{j+1},T))}
&\le  (\bar{A}\bar{B}_j\beta_j^{\mu_3})^{\frac1{\beta_j}}\Big( \| \bar u\|_{L^{\beta_{j+1}}(U\times(t_j,T))}^{\tilde r_j}+\| \bar u\|_{L^{\beta_{j+1}}(U\times(t_j,T))}^{\tilde s_j}\Big)^\frac 1{\beta_{j}},\label{ukbb}
\end{align}
where
$\tilde r_j=\nu_{3,\beta_{j}}$, $\tilde s_j=\nu_{4,\beta_{j}}$, see \eqref{nns}, number $\bar A$ is given in \eqref{Abound}, and 
\beqs
\bar B_j=\chi_*^{1/\hat\kappa} \bar E_j^{1/\hat\kappa} (1+ |Q_T|)^{1+\mu_2}\Big(1+\frac1{t_{j+1}-t_j}\Big)^{\mu_2} \mathcal{M}_{t_j,T}(q_1,q_2,q_3,q_4)^{\mu_2}
\eeqs 
with
$\bar  E_j=\max\left\{ 1, \esssup_{t\in(t_{j+1},T)}\|1+|\Psi(\cdot,t)|\|_{L^\frac{2 r_*}{1-r_*}}^2  + \esssup_{t\in(t_{j+1},T)}\|\bar{u}(\cdot,t)\|_{L^\frac{2 r_*}{1-r_*}}^2\right\}$.

Note from \eqref{mm2} and \eqref{mm3} that $1+\mu_2=\mu_3$. Clearly, $\mathcal{M}_{t_j,T}(q_1,q_2,q_3,q_4)\le \mathcal M_1$, and 
comparing $\bar E_j$ with $E_*$ in \eqref{Estar} gives $\bar E_j\le E_*$.
Let 
\beq\label{Mone}
M_3=\chi_*^{1/\hat\kappa} E_*^{1/\hat\kappa} (1+ |Q_T|)^{\mu_3} \mathcal M_1^{\mu_2} .
\eeq 
Then one can estimate
\begin{align*}
    \bar{A}\bar B_j \beta_j^{\mu_3}
    &\le \bar A M_3\Big(1+\frac1{t_{j+1}-t_j}\Big)^{\mu_2}  \beta_j^{\mu_3}
    =\bar A M_3\Big(1+\frac{2^{j+1}}{\sigma T}\Big)^{\mu_2} (\tilde{\kappa}^j\alpha_0)^{\mu_3}\\
    &\le \bar A M_3 2^{\mu_2(j+1)}\Big(1+\frac{1}{\sigma T}\Big)^{\mu_2} (\tilde{\kappa}^j\alpha_0)^{\mu_3}.
\end{align*}
This yields
\beq\label{ABjbj}
    \bar{A}\bar B_j \beta_j^{\mu_3}
    \le A_{T,\sigma,\alpha_0}^{j+1} \text{ for all }j\ge 0,
\eeq
where 
\beq\label{AAA}
A_{T,\sigma,\alpha_0}=2^{\mu_2}\tilde{\kappa}^{\mu_3} \alpha_0^{\mu_3}\bar A M_3\Big(1+\frac{1}{\sigma T}\Big)^{\mu_2} >1.
\eeq

For $j\ge 0$, define $Y_j=\| \bar u\|_{L^{\beta_{j+1}}(U\times (t_j,T))}.$
Note that $\kappa'\beta_j=\Tilde{\kappa}^2\beta_j=\beta_{j+2}$. Then we have,  by \eqref{ukbb} and \eqref{ABjbj},  
\beq\label{YY}
Y_{j+1}\le A_{T,\sigma,\alpha_0}^{\frac{j+1}{\beta_j}}\big(Y_j^{\tilde{r}_j}+Y_j^{\tilde{s}_j}\big)^{\frac 1{\beta_j}}.
\eeq
Hence, we obtain inequality \eqref{yineq} in Lemma \ref{Genn} for the sequence $(y_j)_{j=0}^\infty=(Y_j)_{j=0}^\infty$.

We check other conditions in Lemma \ref{Genn}.
Because $\tilde\kappa>1$, we have
\beq\label{osum}
\sum_{j=0}^\infty \frac {j+1} { \beta_j}= \frac 1{ \alpha_0}\sum_{j=0}^\infty \frac {j+1} {\tilde \kappa^j}
= \frac{\tilde \kappa^2}{ \alpha_0(\tilde \kappa-1)^2}
\eqdef \ell_1  \in (0,\infty).
\eeq

Using the definitions in \eqref{nns} and the fact $\beta_j\ge \alpha_0>2$, see \eqref{bej2} and \eqref{xal}, we have
\beqs
\frac{\tilde r_j}{\beta_j}=\frac{\beta_j-2}{\beta_j}\in (0,1) \text{ and  }
\frac{\tilde s_j}{\beta_j}
= \left(1+\frac{4(1-a)}{\beta_j}\right) \left(1+\frac{a}{\beta_j-a}\right)\in(1,\infty).
\eeqs
Then
\beqs
0<-\sum_{j=0}^\infty\ln{\frac{\tilde r_j}{\beta_j}}=\sum_{j=0}^\infty\ln{\frac{\beta_j}{\tilde r_j}}
=\sum_{j=0}^\infty\ln\left(1+\frac{2}{\tilde\kappa^j\alpha_0-2}\right)
\le\sum_{j=0}^\infty\frac{2}{\tilde\kappa^j\alpha_0-2}<\infty,
\eeqs
\begin{align*}
0<\sum_{j=0}^\infty\ln{\frac{\tilde s_j}{\beta_j}}
&=\sum_{j=0}^\infty\ln\left(1+\frac{4(1-a)}{\tilde\kappa^j\alpha_0}\right)+\sum_{j=0}^\infty\ln \left(1+\frac{a}{\tilde\kappa^j\alpha_0-a}\right)\\
&\le\sum_{j=0}^\infty\frac{4(1-a)}{\tilde\kappa^j\alpha_0}+\sum_{j=0}^\infty\frac{a}{\tilde\kappa^j\alpha_0-a}
<\infty.
\end{align*}
Therefore, $\sum_{j=0}^\infty\ln(\tilde r_j/\beta_j)=\ell_2\in\R$ and $\sum_{j=0}^\infty\ln(\tilde s_j/\beta_j)=\ell_3\in\R$. Consequently, 
\beq \label{mnutil}
\tilde \mu \eqdef   \prod_{j=0}^\infty \frac{\tilde r_j}{\beta_j}=e^{\ell_2}\text{ and } \tilde \nu \eqdef \prod_{j=0}^\infty \frac{\tilde s_j}{\beta_j}=e^{\ell_3}
\text{ are  positive numbers. }
\eeq 

By \eqref{YY}, \eqref{osum} and \eqref{mnutil}, we can apply Lemma \ref{Genn} to the sequence $(Y_j)_{j=0}^\infty$, and obtain from \eqref{doub2} that
\beq\label{limY}
\limsup_{j\to\infty} Y_{j}\le (2A_{T,\sigma,\alpha_0})^\omega\max\{Y_0^{\tilde\mu}, Y_0^{\tilde\nu}\},
\eeq
where 
\beq\label{oG} \omega=\ell_4 \ell_1 \text{ with } \ell_4=\limsup_{j\to\infty}\left(\max\left\{1, \frac{\tilde s_m}{\beta_m}\cdot \frac{\tilde s_{m+1}}{\beta_{m+1}}\cdots\frac{\tilde s_{m'}}{\beta_{m'}}:1\le m\le m'< j\right\}\right).
 \eeq 
In fact, we have, thanks to the property $\tilde s_j/\beta_j>1$,  that
$$\ell_4=\prod_{k=1}^\infty \frac{\tilde s_k}{\beta_k}=\frac{\tilde\nu \beta_0}{\tilde s_0}\in(0,\infty).$$

Note that $\limsup_{j\to\infty} Y_{j}=\|\bar u\|_{L^\infty(U\times(\sigma T,T))}$, $Y_0=\|\bar  u\|_{L^{\beta_1}(U\times(0,T))}$ and, by \eqref{AAA} and \eqref{Mone},
\beq\label{Aom}
\begin{aligned}
2A_{T,\sigma,\varphi}&=2^{1+\mu_2}\tilde\kappa^{\mu_3}\alpha_0^{\mu_3}\bar{A}(1+ |Q_T|)^{\mu_3}\Big(1+\frac{1}{\sigma T}\Big)^{\mu_2} \chi_*^{1/\hat\kappa} E_*^{1/\hat\kappa} \mathcal M_1^{\mu_2} \\
&=
\bar c_2 \alpha_0^{\mu_3}(1+|Q_T|)^{\mu_3}\Big(1+\frac1{\sigma T}\Big)^{\mu_2}
\chi_*^{1/\hat\kappa}E_*^{1/\hat\kappa} \mathcal M_1^{\mu_2}.
\end{aligned}
\eeq 
Then estimate \eqref{maxineq} follows \eqref{limY} and \eqref{Aom}.
\end{proof}

Combining Theorem \ref{maxest} with the $L^\alpha$-estimate in Theorem \ref{aprio1}, we have the following $L^\infty$-estimate in terms of initial and boundary data, at least for small time.

\begin{theorem}\label{thm45}
Under Assumption \ref{asmp44}, let $\alpha_0$ be a positive number such that 
\beq\label{xa2}
  \alpha_0 \ge\max\left\{ \frac{2p_3(1-a)}{\tilde\kappa-p_3}, \frac{4(1-a)}{\tilde\kappa -1},\frac{a}{1+\lambda_0-\tilde\kappa^2}, \frac{2r_*}{\tilde\kappa(1-r_*)} \right\} 
  \text{ and }
  \alpha_0 >\max\left\{2,\frac{a}{\lambda_0}, \frac{4-3a}{\lambda_0\tilde\kappa} \right\}.
 \eeq 
 
Let $\omega$ and $\tilde \nu$ be the same constants as in Theorem \ref{maxest}.
Denote
\beq \label{Com}
\begin{aligned}
&\beta_1=\tilde \kappa\alpha_0,\
\omega_1=\omega/\hat\kappa,\  
\omega_2=\mu_2\omega,\ 
\omega_3=\mu_3\omega,\\
&\omega_4=\omega_3+\frac{\tilde\nu}{\beta_1} 
\text{ and } \omega_5=\omega_3+\frac{\omega_1(1-r_*)}{r_*}.
\end{aligned}
\eeq 

Let $\mathcal{M}_2(t)=\mathcal{M}_{0,t}(q_1,q_2,q_3,q_4)$ be defined as in \eqref{Mzero}, and $\gamma,C_*,V_0,\mathcal E(\cdot)$ be defined as in Theorem \ref{aprio1} for  $\alpha=\beta_1$.

Suppose  $T>0$ satisfies \eqref{Bdef} for $\alpha=\beta_1$ and some number $B\in(0,1)$. 
For $t\in[0,T]$, let 
\beq\label{VB}
\mathcal V(t)= V_0\left( 1  - \gamma C_*V_0^\gamma \int_0^t \mathcal E(\tau) d\tau  \right)^{-1/\gamma}
\text{ and }
 {\mathscr B}_\sigma(t)=1+\esssup_{\tau\in(\sigma t/2,t)}\|\Psi(\cdot,\tau)\|_{L^\frac{2 r_*}{1-r_*}}.
\eeq

Then one has, for any $t\in(0,T]$ and $\sigma\in(0,1)$, that 
\beq\label{Lb1}
\|\bar u\|_{L^{\infty}(U\times(\sigma t,t))}\le \bar C_1\chi_*^{\omega_1}
(1+\sigma^{-1}t^{-1})^{\omega_2}(1+t)^{\omega_4} 
\mathcal M_2(t)^{\omega_2} {\mathscr B}_\sigma(t)^{2\omega_1}\mathcal V(t)^\frac{2\omega_1+\tilde\nu}{\beta_1},
\eeq
where 
$\bar C_1=2^{\omega_1}\bar c_2^\omega \alpha_0^{\omega_3}(1+|U|)^{\omega_5}$.
\end{theorem}
\begin{proof}
Thanks to condition \eqref{xa2},  $\alpha_0$ satisfies \eqref{xal}, and $\alpha=\beta_1$ satisfies  \eqref{alph1}. 
Let $\tilde \mu$ be the number as in Theorem \ref{maxest}.
Applying estimate \eqref{maxineq} to $T:=t$ and using definitions of constants in \eqref{Com}, we have 
\beq\label{Li2}
\begin{split}
\|\bar u\|_{L^{\infty}(U\times(\sigma t,t))}
&\le C_3\chi_*^{\omega_1}(1+\sigma^{-1}t^{-1})^{\omega_2}(1+|Q_t|)^{\omega_3} \mathcal M_2^{\omega_2}
E_*(t)^{\omega_1}  \\
&\quad\times \max\left\{  \|\bar u\|_{L^{\beta_1}(U\times(0,t))}^{\tilde\mu} , \|\bar u\|_{L^{\beta_1}(U\times(0,t))}^{\tilde\nu}\right\},
\end{split}
\eeq
where 
$C_3=(\bar c_2\alpha_0^{\mu_3})^{\omega}$ and
\beqs
 E_*(t)=\max\left\{ 1, \esssup_{\tau\in(\sigma t/2,t)}\|1+|\Psi(\cdot,\tau)|\|_{L^\frac{2 r_*}{1-r_*}}^2  + \esssup_{\tau\in(\sigma t/2,t)}\|\bar{u}(\cdot,\tau)\|_{L^\frac{2 r_*}{1-r_*}}^2\right\}.
\eeqs

Note that $\mathcal V(t)$ is increasing in $t\in[0,T]$.
By \eqref{u2}, we have, for all $\tau\in[0,t]$, 
\beq\label{Lb0}
 \int_U |\bar u(x,\tau)|^{\beta_1}dx  \le \mathcal V(\tau)\le \mathcal V(t).
\eeq 
Hence,
\beqs 
 \int_0^t\int_U |\bar u(x,\tau)|^{\beta_1}dxd\tau  \le t\mathcal V(t).
\eeqs 
Combining this estimate with the facts $\tilde\nu>\tilde \mu$ and $\mathcal V(t)\ge 1$ yields
\beq\label{Lb}
\begin{aligned}
\max\left\{  \|\bar u\|_{L^{\beta_1}(U\times(0,t))}^{\tilde\mu} , \|\bar u\|_{L^{\beta_1}(U\times(0,t))}^{\tilde\nu}\right\}
&\le \max\Big\{ (t\mathcal V(t))^{\frac{\tilde\mu}{\beta_1}},(t\mathcal V(t))^\frac{\tilde\nu}{\beta_1}\Big\}\\
&\le (1+t)^{\frac{\tilde\nu}{\beta_1}}\mathcal V(t)^\frac{\tilde\nu}{\beta_1}.
\end{aligned}
\eeq 

For $E_*(t)$, we, on the one hand, use the triangle inequality to estimate
\beqs 
\esssup_{\tau\in(\sigma t/2,t)} \|1+|\Psi(\cdot,\tau)|\|_{L^\frac{2 r_*}{1-r_*}}^2 
\le \esssup_{\tau\in(\sigma t/2,t)} \left(|U|^{\frac{1-r_*}{2r_*}} +\|\Psi(\cdot,\tau)\|_{L^\frac{2 r_*}{1-r_*}}\right)^2
 \le (1+|U|)^{\frac{1}{r_*}-1}{\mathscr B}_\sigma(t)^2.
\eeqs 
On the other hand, we use H\"older's inequality and \eqref{Lb0} to obtain
\beqs 
\esssup_{\tau\in(\sigma t/2,t)}\|\bar{u}(\cdot,\tau)\|_{L^\frac{2 r_*}{1-r_*}}^2 
\le |U|^{\frac{1}{r_*}-1-\frac{2}{\beta_1}} \esssup_{\tau\in(\sigma t/2,t)}\|\bar{u}(\cdot,\tau)\|_{L^{\beta_1}}^2
 \le (1+|U|)^{\frac{1}{r_*}-1}\mathcal V(t)^{2/\beta_1}.
\eeqs
Hence,
\beq\label{widetildeE}
 E_*(t)  \le 2(1+|U|)^{\frac{1}{r_*}-1}{\mathscr B}_\sigma(t)^2\mathcal V(t)^{2/\beta_1}.
\eeq

Combining \eqref{Li2}, \eqref{Lb} and \eqref{widetildeE} with the fact $1+|Q_t|\le (1+|U|)(1+t)$, we have 
\beqs
\begin{split}
\|\bar u\|_{L^{\infty}(U\times(\sigma t,t))}
&\le C_3\chi_*^{\omega_1}(1+\sigma^{-1}t^{-1})^{\omega_2}(1+t)^{\omega_3} (1+|U|)^{\omega_3}  \mathcal M_2(t)^{\omega_2}
\\
&\quad\times \left[2(1+|U|)^{\frac{1}{r_*}-1}{\mathscr B}_\sigma(t)^2 \mathcal V(t)^{2/\beta_1} \right]^{\omega_1} (1+t)^{\frac{\tilde\nu}{\beta_1}}\mathcal V(t)^\frac{\tilde\nu}{\beta_1}\\
&= 2^{\omega_1}C_3(1+|U|)^{\omega_5}\chi_*^{\omega_1}(1+\sigma^{-1}t^{-1})^{\omega_2}(1+t)^{\omega_4}  \mathcal M_2(t)^{\omega_2}
{\mathscr B}_\sigma(t)^{2\omega_1} \mathcal V(t)^\frac{2\omega_1+\tilde\nu}{\beta_1}.
\end{split}
\eeqs
Then inequality \eqref{Lb1} follows.
\end{proof}

\section{Maximum principle}\label{maxprin}

In this section, we estimate the classical solutions of \eqref{ibvpg} by the maximum principle.
Recall that the functions $X(z,y)$, $\mathcal Z(x,t)$ and $\Phi(x,t)$ are defined by \eqref{Xdef}, \eqref{Zx1} and \eqref{Phi00}, respectively.
We re-write equation \eqref{mainuX} in the non-divergence form as
\beq\label{unondiv}
 u_t=D_y X(u,\Phi):(D^2 u + u^2 \Omega^2 \mathbf J^2 + 2u  \nabla u \mathcal Z^{\rm T})+D_z X(u,\Phi)\cdot \nabla u.
\eeq

For $T>0$, denote $U_T=U\times (0,T]$, its closure $\overline{U_T}=\overline U\times[0,T]$ and its parabolic boundary  
$\partial_p U_T=\overline{U_T}\setminus U_T=U\times\{0\}\cup \Gamma\times[0,T]$.

\begin{theorem}\label{maxpr}
Assume $u\in C(\overline{U_T})\cap C_{x,t}^{2,1}(U_T)$, $u\ge 0$ on $\overline{U_T}$ and $u$ satisfies \eqref{mainuX} in $U_T$. Then one has
 \beq\label{max-u}
 \max_{\overline{U_T}} u =\max_{\partial_p U_T} u.
 \eeq
\end{theorem} 
\begin{proof} Given any $\varep>0$, let  $u^\varep(x,t)=e^{-\varep t}u(x,t)$ and 
$\displaystyle M_\varep=\max_{\overline{U_T}} u^\varep$.
We claim that
 \beq\label{maxe}
M_\varep = \max_{\partial_p U_T} u^\varep.
 \eeq

Suppose \eqref{maxe} is false. Then $M_\varep>0$ and there exists a point $(x_0,t_0)\in U_T$ such that 
$u^\varep(x_0,t_0)=M_\varep$. At this maximum point $(x_0,t_0)$ we have 
\beq\label{vt1}
u^\varep_t (x_0,t_0)\ge 0\text{ and }
Du^\varep(x_0,t_0)=0.
\eeq

It is proved in \cite[Theorem 3.1]{CHK3}, based mainly on property \eqref{hXh} in Lemma \ref{Xder},  that
\beq\label{old}
D_y X(u,\Phi):(D^2 u + u^2 \Omega^2 \mathbf J^2)\Big|_{(x,t)=(x_0,t_0)}\le 0.
\eeq

The second property of \eqref{vt1} deduces $Du(x_0,t_0)=0$. This fact, \eqref{old} and \eqref{unondiv} imply $u_t(x_0,t_0)\le 0$.
Therefore,
\beqs
u^\varep_t(x_0,t_0)=-\varep u^\varep(x_0,t_0) +e^{-\varep t_0} u_t(x_0,t_0)
\le -\varep M_\varep <0,
\eeqs
which contradicts the first inequality in \eqref{vt1}. Thus, \eqref{maxe} holds true.
 Note that
 \beqs
e^{-\varep T}\max_{\overline{U_T}} u\le M_\varep = \max_{\partial_p U_T} u^\varep\le \max_{\partial_p U_T} u
\le \max_{\overline{U_T}} u.
 \eeqs
Then passing $\varep\to 0$, we obtain \eqref{max-u}.
\end{proof}

In the following,  $T_*\in(0,\infty]$ is fixed.

Clearly, if $u\in C(\overline U\times[0,T_*))\cap C_{x,t}^{2,1}(U\times (0,T_*))$ is a nonnegative solution of problem \eqref{ibvpg},
then, by the virtue of  Theorem \ref{maxpr}, we have the maximum estimates in terms of the initial and boundary data: 
\beq\label{uM} 
\sup_{x\in U}u(x,t)\le \max\left\{\sup_{x\in U} u_0(x),\sup_{(x,\tau)\in \Gamma \times (0,t]} \psi(x,\tau)\right\}\text{ for all $t\in(0,T_*)$.}
\eeq

In case the solution $u$ belongs to $C(\overline U\times(0,T_*))$ but not $C(\overline U\times[0,T_*))$, estimate \eqref{uM} is not applicable.
For instance, initial data $u_0$ is unbounded.
However, under certain weaker conditions, the maximum estimates can be still be established by combining Theorems \ref{thm45} and \ref{maxpr}. 

Under Assumption \ref{asmp44}, let $\alpha_0$ satisfy \eqref{xa2}.
We use the same notation as in Theorem \ref{thm45}.
Assume further that 
\begin{enumerate}[label=\tnum]
    \item\label{m2} $\mathcal M_2(t)$ is finite for all $t\in (0,T_*)$ and $\mathcal E\in L^1_{\rm loc}([0,T_*))$,
    \item\label{Psc} $\Psi\in C(\overline U\times(0,T_*))\cap C([0,T_*),L^{\beta_1}(U))$.
\end{enumerate}

Because of the second property in \ref{m2}, we can find a number $t_0$ such that $0<t_0\le 1$, $t_0<T_*$, and \eqref{Bdef} is satisfied for $T=t_0$, $\alpha=\beta_1$, and some number $B\in(0,1)$.
 
\begin{theorem}\label{maxestsol} 
Let $u\in C(\overline U\times(0,T_*))\cap C_{x,t}^{2,1}(U\times (0,T_*))\cap C([0,T_*),L^{\beta_1}(U))$ be a nonnegative solution of problem \eqref{ibvpg}. 

If $t\in(0,t_0]$, then 
\beq\label{max1}
  \begin{aligned}
 \sup_{x\in U} u(x,t) \le  \bar C_2\chi_*^{\omega_1}t^{-\omega_2}
\mathcal M_2(t)^{\omega_2} {\mathscr B}_*(t)^{2\omega_1}\mathcal V(t)^\frac{2\omega_1+\tilde\nu}{\beta_1}
+ \sup_{x\in U} |\Psi(x,t)|,
\end{aligned}
\eeq
where $\bar C_2= 3^{\omega_2} 2^{\omega_4}\bar C_1$ and 
\beq\label{Bstar}
 {\mathscr B}_*(t)=1+\esssup_{\tau\in(t/4,t)}\|\Psi(\cdot,\tau)\|_{L^\frac{2 r_*}{1-r_*}}.
\eeq

If $t\in(t_0,T_*)$ then
  \beq\label{ubd}
 \sup_{x\in U} u(x,t) \le  \max\left\{\bar C_3\chi_*^{\omega_1}\mathcal M_3(t_0)\left(1+\|\bar u_0\|_{L^{\beta_1}}\right)^{2\omega_1+\tilde\nu}
 + \sup_{x\in U} |\Psi(x,t_0)|,  \sup_{(x,\tau)\in \Gamma \times[t_0,t]} \psi(x,\tau)\right\},
\eeq
where $\bar C_3= \bar C_2 (1-B)^{-\frac{2\omega_1+\tilde\nu}{\beta_1\gamma}}$ and 
$\mathcal M_3(t_0)=t_0^{-\omega_2}\mathcal M_2(t_0)^{\omega_2}{\mathscr B}_*(t_0)^{2\omega_1}$.
\end{theorem}
\begin{proof}
First, note that ${\mathscr B}_*(t)$ in \eqref{Bstar} is, in fact, ${\mathscr B}_{1/2}(t)$ in \eqref{VB}.
Let $t\in(0,t_0]$.  By estimate \eqref{Lb1} applied to  $\sigma=1/2$, we have 
\begin{align*}
\|\bar u\|_{L^{\infty}(U\times(t/2,t))}
&\le \bar C_1\chi_*^{\omega_1}
(1+2t^{-1})^{\omega_2}(1+t)^{\omega_4} 
\mathcal M_2(t)^{\omega_2} {\mathscr B}_*(t)^{2\omega_1}\mathcal V(t)^\frac{2\omega_1+\tilde\nu}{\beta_1}\\
&\le \bar C_1\chi_*^{\omega_1}
(3t^{-1})^{\omega_2}2^{\omega_4} 
\mathcal M_2(t)^{\omega_2} {\mathscr B}_*(t)^{2\omega_1}\mathcal V(t)^\frac{2\omega_1+\tilde\nu}{\beta_1}\\
&=\bar C_2\chi_*^{\omega_1} t^{-\omega_2} 
\mathcal M_2(t)^{\omega_2} {\mathscr B}_*(t)^{2\omega_1}\mathcal V(t)^\frac{2\omega_1+\tilde\nu}{\beta_1}.
\end{align*}

By the continuity of $\bar u$ on $\overline U\times [t/2,t]$,
\begin{align*}
 \sup_{x\in U} |\bar u(x,t)| 
 \le \|\bar u\|_{L^{\infty}(U\times(t/2,t))}
\le \bar C_2\chi_*^{\omega_1} t^{-\omega_2} 
\mathcal M_2(t)^{\omega_2} {\mathscr B}_*(t)^{2\omega_1}\mathcal V(t)^\frac{2\omega_1+\tilde\nu}{\beta_1}.
\end{align*}

Combining this estimate with the triangle inequality $u(x,t)\le |\bar u(x,t)|+|\Psi(x,t)|$ gives \eqref{max1}.

Let $t\in(t_0,T_*)$ now.  Applying the maximum principle in Theorem~\ref{maxpr} for the interval $[t_0,t]$ in place of $[0,T]$, we have  
\beq\label{umax}
 \sup_{x\in U} u(x,t) \le \max\left\{ \sup_{x\in U} u(x,t_0),  \sup_{(x,\tau)\in \Gamma\times [t_0,t]} \psi(x,\tau)\right\}.
 \eeq

Estimating $u(x,t_0) $ by using \eqref{max1} for $t=t_0$ yields
 \beq\label{ut0}
 \sup_{x\in U} u(x,t_0) \le \bar C_2\chi_*^{\omega_1} t_0^{-\omega_2} 
\mathcal M_2(t_0)^{\omega_2} {\mathscr B}_*(t_0)^{2\omega_1}\mathcal V(t_0)^\frac{2\omega_1+\tilde\nu}{\beta_1}
+ \sup_{x\in U} |\Psi(x,t_0)|.
 \eeq

Thanks to estimate \eqref{u2} for $t=t_0$,  we have 
\beq\label{Vt0}
\mathcal V(t_0)^{1/\beta_1}\le V_0^{1/\beta_1} (1-B)^{-1/(\beta_1\gamma)}.
\eeq

Applying inequality \eqref{ee3} to $x=1$, $y=\int_U |u_0(x)|^{\beta_1}dx$ and  $p=1/\beta_1<1$ gives
\beq\label{Vzuz} V_0^{1/\beta_1}\le 1+\|u_0\|_{L^{\beta_1}}.
\eeq 

Then estimate \eqref{ubd} follows from \eqref{umax}, \eqref{ut0}, \eqref{Vt0} and \eqref{Vzuz}.
\end{proof}

\begin{remark}
The following final remarks are in order.
\begin{enumerate}[label=\rnum]
    \item As a sequel of our previous work \cite{CHK3}, the current paper only considers slightly compressible fluids. 
Nonetheless, the methods developed here and in \cite{CHK1,CHK2} can be applied to analyze other types of (compressible) gaseous flows in rotating porous media.
   \item Our analysis can be easily adapted for more general PDE of type \eqref{mainuX} in space $\R^n$ not just $\R^3$. The function $X$ is only required to have similar properties to those in Lemmas \ref{lem21} and \ref{Xder}.
\end{enumerate}
\end{remark}

\def\cprime{$'$}

\end{document}